\documentclass[reqno,11pt]{amsart}

\usepackage{amsthm, amsmath, amssymb, stmaryrd}
\usepackage[utf8]{inputenc}
\usepackage[T1]{fontenc}

\usepackage[hyphenbreaks]{breakurl}
\usepackage[hyphens]{url}

\usepackage[shortlabels]{enumitem}
\usepackage[hidelinks]{hyperref}
\usepackage{microtype}

\usepackage{bm}
\usepackage[margin=1in]{geometry}

\usepackage[textsize=scriptsize,backgroundcolor=orange!5]{todonotes}

\usepackage[noabbrev,capitalize,sort]{cleveref}
\crefname{equation}{}{}
\numberwithin{equation}{section}

\newtheorem{theorem}{Theorem}[section]

\newtheorem{lemma}[theorem]{Lemma}

\newtheorem{conjecture}[theorem]{Conjecture}

\theoremstyle{definition}
\newtheorem{definition}[theorem]{Definition}

\theoremstyle{remark}
\newtheorem*{remark}{Remark}

\newcommand{\abs}[1]{\left\lvert#1\right\rvert}
\newcommand{\sabs}[1]{\lvert#1\rvert}

\newcommand{\ang}[1]{\left\langle #1 \right\rangle}

\newcommand{\paren}[1]{\left( #1 \right)}

\newcommand{\sqbb}[1]{\left\llbracket #1 \right\rrbracket}
\newcommand{\set}[1]{\left\{ #1 \right\}}

\newcommand{\wh}{\widehat}

\DeclareMathOperator{\codim}{codim}
\DeclareMathOperator{\Span}{span}

\newcommand{\FF}{\mathbb{F}}
\newcommand{\RR}{\mathbb{R}}
\newcommand{\NN}{\mathbb{N}}
\newcommand{\ZZ}{\mathbb{Z}}

\newcommand{\cL}{\mathcal L}
\newcommand{\cJ}{\mathcal J}
\newcommand{\cP}{\mathcal P}
\newcommand{\cF}{\mathcal F}
\newcommand{\cV}{\mathcal V}

\newcommand{\cD}{\mathcal D}

\newcommand{\cM}{\mathcal M}
\newcommand{\cB}{\mathcal B}
\newcommand{\cS}{\mathcal S}

\newcommand{\Hasse}{\mathsf H} 

\newcommand{\totalB}{\mathbb B}
\newcommand{\prioB}{\mathcal B}
\newcommand{\totalD}{\mathbb D}
\newcommand{\prioD}{\mathcal D}

\newcommand{\totalT}{\mathbb T}

\author[Tidor]{Jonathan Tidor}
\author[Yu]{Hung-Hsun Hans Yu}
\author[Zhao]{Yufei Zhao}

\address{Tidor, Zhao: Department of Mathematics, Massachusetts Institute of Technology, Cambridge, MA, USA}
\address{Yu:   Trinity College, University of Cambridge, UK}
\email{\{jtidor,hansonyu,yufeiz\}@mit.edu}

\title{Joints of varieties}

\begin{document}

\begin{abstract}
	We generalize the Guth--Katz joints theorem from lines to varieties.
	A special case says that $N$ planes (2-flats) in 6 dimensions (over any field) have $O(N^{3/2})$ joints, where a joint is a point contained in a triple of these planes not all lying in some hyperplane. More generally, we prove the same bound when the set of $N$ planes is replaced by a set of 2-dimensional algebraic varieties of total degree $N$, and a joint is a point that is regular for three varieties whose tangent planes at that point are not all contained in some hyperplane. Our most general result gives upper bounds, tight up to constant factors, for joints with multiplicities for several sets of varieties of arbitrary dimensions (known as Carbery's conjecture). Our main innovation is a new way to extend the polynomial method to higher dimensional objects, relating the degree of a polynomial and its orders of vanishing on a given set of points on a variety.
\end{abstract}

\maketitle
\section{Introduction}

Guth and Katz~\cite{GK10} proved the following ``joints theorem'': $N$ lines in $\RR^3$ have $O(N^{3/2})$ joints, where a \emph{joint} is a point contained in three of the lines that do not all lie on some plane.
This bound is tight up to a constant factor due to the following example: consider $k$ generic planes---their pairwise intersections give $\binom{k}{2}$ lines and triplewise intersections give $\binom{k}{3}$ joints. 

The joints problem was first studied in Chazelle et al.~\cite{CEGPSSS92}.
Besides being an interesting problem in incidence geometry,
it also caught the attention of harmonic analysts due to connections to the Kakeya problem as observed by Wolff~\cite{Wol99}.
This connection was further elucidated by Bennett, Carbery and Tao~\cite{BCT06} in their work on the multilinear Kakeya problem, which in turn allowed them to improve bounds on the joints problem (prior to the Guth--Katz solution).
Guth~\cite{Guth10} later adapted techniques from the solution of the joints theorem to prove the so-called endpoint case of the Bennett--Carbery--Tao multilinear Kakeya conjecture, which can be viewed as a joints theorem for tubes (also see the exposition in \cite[Section 15.8]{Guth-book}). Guth's multilinear Kakeya result was later generalized by Zhang~\cite{Zha18} to slabs and neighborhoods of varieties (though the latter does not translate back to the joints problem for flats).

The Guth--Katz solution of the joints problem highlights the importance of the  polynomial method. 
Their joints theorem was also a precursor to their subsequent breakthrough on the Erd\H{o}s distinct distances problem~\cite{GK15}, which introduced a polynomial partitioning method that has found many subsequent applications.
One of the key steps in \cite{GK15} dealt with a point-line incidence problem in $\RR^3$ with additional constraints on the configuration of lines.
These developments were partly inspired by Dvir's~\cite{Dvir09} stunningly short and elegant solution to the finite field Kakeya problem. Guth has also successfully applied the polynomial method  developed in this line of work to restriction problems related to Kakeya~\cite{Guth16,Guth18}.

Since Guth and Katz's original work, there has been significant effort in extending the joints theorem~\cite{CI14,CI20,CV14,EKS11,Hab,Ili13,Ili-thesis,Ili15a,Ili15b,KSS10,Qui09,YZ,Zha20}.
Kaplan, Sharir, and Shustin~\cite{KSS10} and Quilodr\'an~\cite{Qui09} independently extended the joints theorem from $\RR^3$ to $\RR^d$, and these techniques and results extend to arbitrary fields as stated below (also see~\cite{CI14,Dvir10,Tao14}). Given a set of lines in $\FF^d$, a \emph{joint} is a point contained in $d$ lines with independent and spanning directions. Throughout the paper, $\FF$ stands for an arbitrary field, and our constants do not depend on $\FF$.

\begin{theorem} \label{thm:joints-lines}
	A set of $N$ lines in $\FF^d$ has at most $C_d N^{d/(d-1)}$ joints, for some constant $C_d$.
\end{theorem}



Recently Yu and Zhao~\cite{YZ} proved that $N$ lines in $\FF^d$ have at most $\tfrac{(d-1)!^{1/(d-1)}}{d}N^{d/(d-1)}$ joints. This leading constant is optimal, matching the above construction up to a ($1+o(1)$)-factor.

We generalize the joints theorem from lines to varieties, overcoming a fundamental difficulty with the polynomial method that one quickly runs into---we will elaborate more on this later.
A representative case of our result says the following.
Here a \emph{joint} is a point contained in a triple of planes not all lying in some hyperplane.
All our bounds on joints in this paper are tight up to a constant factor (depending on the dimension) due to a straightforward generalization of the example in the first paragraph.

\begin{theorem} \label{thm:2-flats}
	A set of $N$ planes in $\FF^6$ has $O(N^{3/2})$ joints.
\end{theorem}

In his PhD thesis, Ben Yang~\cite{Yang,Yangthesis} proved partial results giving an upper bound $N^{3/2+o(1)}$ when $\FF = \RR$ 
(and also more generally for bounded degree varieties in $\RR^d$---in contrast, our results on joints of varieties do not require any bounded degree hypotheses).
Yang's results have two fundamental limitations: 
(1) an error term in the exponent 
and 
(2) the methods only work over the reals.
He used a variant of the polynomial partitioning method~\cite{GK15}, which requires real topology.
More specifically, Yang applied polynomial partitioning for varieties (due to Guth~\cite{Guth15} and extended by Blagojevi\'c, Blagojevi\'c, and Ziegler~\cite{BBZ17}) using bounded degree polynomials (due to Solymosi and Tao~\cite{ST12}), 
with the latter requiring an error term in the exponent.
We introduce a novel approach that avoids both limitations.

The only other prior result on joints of higher dimensional objects says that, as a representative example, a set of $L$ lines and $F$ planes in $\FF^4$ has $O(LF^{1/2})$ joints, where now a joint is defined to be a point contained in two lines and one plane, not all lying on a hyperplane (this result was recently independently proved by Yu and Zhao~\cite{YZ} and Carbery and Iliopoulou~\cite{CI20}; Yang mentioned at the end of his thesis~\cite{Yangthesis} that he could also obtain this claim, though without details). 
Even the ``next'' case of ``line-plane-plane'' joints was open before this work. 

Incidence geometry and the polynomial method concerning higher dimensional objects often tend to be substantially more intricate compared to problems that only involve lines and points.
Our work introduces a new way to tackle such problems.
Let us highlight some other representative works on higher dimensional incidence problems.
Solymosi and Tao~\cite{ST12} introduced a bounded degree variation of the polynomial partitioning method, used in Yang's proof mentioned earlier, to give nearly tight (up to a $+o(1)$ error term in the exponent) bound for incidences between points and $k$-dimensional varieties of bounded degree in $\RR^d$, in the spirit of the Szemer\'edi--Trotter theorem~\cite{ST83} for point-line incidences in the plane. 
Using different methods, Walsh~\cite{Wal19,Wal20} recently developed powerful techniques for understanding incidences between sets of $m$-dimensional and $m+1$-dimensional varieties, thereby unifying a large body of incidence geometry results in the literature.
However, we do not see how to apply Walsh's techniques for extending the joints theorem.
The above approaches use different forms of ``partitioning'' and involve iteratively restricting the ambient space to a codimension-1 subvariety, which usually involves an increment in the degree of the ambient variety. By contrast, our strategy does not use any form of partitioning.


The main innovation of our work is a new method of relating degrees and orders of vanishing for multivariate polynomials.
Earlier approaches, e.g., \cite{ST12,Wal19,Wal20,Yang,Zha18}, consider multiple polynomials, and are related to understanding B\'ezout's theorem and possible inverses (see Tao's blog post~\cite{Tao-blog-bezout} on inverse B\'ezout).
Our approach instead only considers a single polynomial via parameter counting but we have to be extremely delicate in choosing vanishing conditions.
We motivate and explain these ideas in \cref{sec:motivation}.
The polynomial method is already a powerful technique in discrete geometry, analysis, number theory, and theoretical computer science, and we hope that our method for handling higher dimensional objects will find additional applications.

The most general version of our result is \cref{thm:car-var} below, and it implies all the other statements. 
Next we gradually introduce the various generalizations and explain the history.
The reader who is only interested in the proof of  \cref{thm:2-flats} can safely skip the rest of this section and proceed to \cref{sec:motivation} and \cref{sec:joints-of-planes} for the key ideas and the proof of \cref{thm:2-flats}.

\subsection{Joints of flats}

We extend \cref{thm:2-flats} to flats of arbitrary dimensions.
Given a collection of $k$-flats (i.e., $k$-dimensional flats) in $\FF^{mk}$, a \emph{joint} is defined to be a point contained in $m$ of these $k$-flats and not all contained in a single hyperplane.

\begin{theorem}\label{thm:flats}
	A set of $N$ $k$-flats in $\FF^{mk}$ has at most $C_{m,k} N^{m/(m-1)}$ joints, for some constant $C_{m,k}$.
\end{theorem}

\subsection{Multijoints}

In the joints problem, instead of a single set of lines in $\FF^d$, we can consider $d$ sets of lines $\cL_1, \dots, \cL_d$ in $\FF^d$ and consider joints formed by taking one line from each $\cL_i$ (each point is counted as a joint at most once, for now). This variation, known as ``multijoints'', can be viewed as a discrete analogue of the endpoint multilinear Kakeya problem.
The following bound on multijoints was conjectured by Carbery, proved in $\FF^3$ and $\RR^d$ by Iliopoulou~\cite{Ili15b} and in general $\FF^d$ by Zhang~\cite{Zha20}.
Note that the the multijoints theorem is equivalent to the joints theorem if $\abs{\cL_i}$ are all within a constant factor of each other.

\begin{theorem}[Multijoints of lines] \label{thm:mult-lines}
    Given $d$ sets of lines $\cL_1, \dots, \cL_d$ in $\FF^d$, the number of joints formed by taking one line from each $\cL_i$ is at most $C_d (\abs{\cL_1}\cdots \abs{\cL_d})^{1/(d-1)}$ for some constant $C_d$.
\end{theorem}

We extend the multijoints theorem from lines to flats. Here a point is a \emph{joint} formed by several flats if these flats contain this point and have spanning and independent directions.

\begin{theorem}[Multijoints of flats]  \label{thm:mult-flats}
Given $\cF_1, \dots, \cF_r$, where $\cF_i$ is a set of $k_i$-flats in $\FF^d$, with $d = k_1 + \cdots + k_r$, the number of joints formed by taking one flat from each $\cF_i$ is at most $C_{k_1, \dots, k_r} (\abs{\cF_1}\cdots \abs{\cF_r})^{1/(r-1)}$ for some constant $C_{k_1, \dots, k_r}$.	
\end{theorem}

\subsection{Varieties}

We extend the joints theorem from flats to varieties. Generalizing earlier notions, a point $p$ is a \emph{joint} formed by several varieties $V_1, \dots, V_r$ if $p$ is a regular point for each $V_i$ and their tangent spaces at $p$ have independent and spanning directions. (Recall that a point $p$ is a \emph{regular point} of a variety $V$ if the Zariski tangent space $T_pV$ has the same dimension as $V$.)

The proof of the joints theorem can be easily adapted from lines to algebraic curves (e.g., see~\cite{KSS10,Qui09}).
Here we extend the joints theorem to higher dimensional varieties. Given a set $\cV$ of varieties, let $\deg \cV$ denote the sum of the degrees of the elements of $\cV$.

\begin{theorem}[Joints of varieties] \label{thm:joints-var}
	A set $\cV$ of $k$-dimensional varieties in $\FF^{mk}$ has at most $C_{m, k} (\deg \cV)^{m/(m-1)}$ joints for some constant $C_{m,k}$.
\end{theorem}

\begin{remark}
In this paper, all varieties are assumed to be irreducible. 
We do not lose any generality for the joints problem with this assumption as one can always replace any algebraic set by its irreducible components.
\end{remark}

Like earlier, we prove the result more generally for multiple sets of varieties.

\begin{theorem}[Multijoints of varieties] \label{thm:mult-var}
Given $\cV_1, \dots, \cV_r$, where each $\cV_i$ is a set of $k_i$-dimensional varieties in $\FF^d$, where $d = k_1 + \cdots + k_r$, the number of joints formed by taking one variety from each $V_i$ is at most $C_{k_1, \dots, k_r} (\deg \cV_1 \cdots \deg \cV_r)^{1/(r-1)}$ for some constant $C_{k_1, \dots, k_r}$.
\end{theorem}

Previously, Iliopoulou~\cite{Ili15b} proved the multijoints theorem for algebraic curves of bounded degree in $\RR^d$ (here by bounded degree we mean that the leading constant $C$ depends on the maximum degree of the curves), but it was unknown how to to generalize from $\RR^d$ to $\FF^d$, despite knowledge of the joints theorem for a single set of curves. 
This is because Zhang's proof~\cite{Zha20} of the multijoints theorem for lines (\cref{thm:mult-lines}) does not easily adapt to curves.

In the setting of real varieties, Yang~\cite{Yang} proved an upper bound of the form $C_{\epsilon}(\abs{\cV_1} \cdots \abs{\cV_r})^{1/(r-1) + \epsilon}$ for all $\epsilon>0$ where $C_{\epsilon}$ also depends on the maximum degree of the varieties.

\subsection{Joints with multiplicities}

In the above formulations of joints and multijoints theorems, each point is counted as a joint at most once. 
Motivated by Kakeya problems, Carbery suggested a generalization where joints contained in many lines are counted with multiplicity. 
The following theorem about joints of lines with multiplicities was conjectured by Carbery, proved in $\RR^3$ by Iliopoulou~\cite{Ili13}, and settled in general by Zhang~\cite{Zha20}. 

\begin{theorem}[Joints of lines with multiplicities] \label{thm:car-lines}
	Let $\cL_1, \dots, \cL_d$ be multisets of lines in $\FF^d$. 
	Let $M(p)$ denote the number of tuples of lines $(\ell_1, \dots, \ell_d) \in \cL_1 \times \cdots \times \cL_d$ that form a joint at $p$. Summing over all such joints $p$, we have
	\[
		\sum_p M(p)^{1/(d-1)} \le C_d (\abs{\cL_1} \cdots \abs{\cL_d})^{1/(d-1)},
	\]
	where $C_d$ is some constant.
\end{theorem}

\cref{thm:car-lines} strengthens \cref{thm:mult-lines} (multijoints of lines). The exponent in $M(p)^{1/(d-1)}$ on the left-hand side is optimal as can be easily seen by duplicating every element in each set of lines $m$ times for some large $m$.

Yang~\cite{Yang} studied a generalization of \cref{thm:car-lines} to joints of varieties with multiplicities, but as earlier, his upper bound only holds in $\RR^d$, carries an $+o(1)$ error term in the exponent, and the leading constant depends on the maximum degree of the varieties.

Our main result, below, generalizes the above to joints of varieties counted with multiplicities. It generalizes all previously stated results.

\begin{theorem}[Joints of varieties with multiplicities]\label{thm:car-var}
	For each $i = 1, \dots, r$, let $\cV_i$ be a multiset of $k_i$-dimensional varieties in $\FF^d$, where $d = k_1 + \cdots + k_r$.
	Let $M(p)$ denote the number of tuples of varieties $(V_1, \dots, V_r) \in \cV_1 \times \cdots \times \cV_r$ that form a joint at $p$. Summing over all such joints $p$, we have
	\[
		\sum_p M(p)^{1/(r-1)} \le C_{k_1,\dots,k_r} (\deg \cV_1 \cdots \deg \cV_r)^{1/(r-1)},
	\]
	where $C_{k_1,\dots,k_r}$ is some constant.
\end{theorem}

Our proof of \cref{thm:car-var} even in the case of lines is different from that of Zhang~\cite{Zha20}. By our method, there is no significant difference between the proofs of \cref{thm:mult-var} (without multiplicities) and \cref{thm:car-var} (with multiplicities).

\subsection{Constants}

We restate \cref{thm:mult-var,thm:car-var} in the following equivalent form with explicit constansts. This superficially more general formulation (formulated in \cite{YZ} for flats) exposes a difficulty hierarchy of the problem. It also allows us to discuss the leading constants. 
While the constants below are optimal for $(r,k_1,m_1) = (1,1,d)$, they are likely not tight in all other cases.

\begin{theorem}[Main theorem]
\label{thm:main} 
Let $k_1, \dots, k_r, m_1, \dots, m_r$ be positive integers. For each $i=1,\dots,r$, let $\cV_i$ be a finite multiset of $k_i$-dimensional varieties in $\FF^d$, where $d = m_1k_1 + \cdots + m_rk_r$. We only consider joints $p$ formed by choosing $m_i$ unordered elements from $\cV_i$ for each $i = 1, \dots, r$, and we write $M(p)$ for the number of such choices.
\begin{enumerate}[(a)]
    \item \label{part:const-no-mult}
    (without multiplicities)  
    The number of joints is at most 
    \[
    C_{k_1, \dots, k_r; m_1, \dots, m_r} \paren{(\deg \cV_1)^{m_1} \cdots (\deg \cV_r)^{m_r}}^{1/(m_1 + \cdots + m_r - 1)},
    \]
    where
	\[
        C_{k_1, \dots, k_r; m_1, \dots, m_r} = \left(\frac{d!}{\prod_{i=1}^{r}k_i!^{m_i}m_i^{m_i}}\right)^{1/(m_1+\cdots+m_r-1)}.
    \]
    \item  \label{part:const-mult}
    (with multiplicities) 
    Summing over all joints $p$, one has
    \[
    \sum_p M(p)^{1/(m_1+\cdots+m_r-1)} \le C'_{k_1, \dots, k_r; m_1, \dots, m_r} \paren{(\deg \cV_1)^{m_1} \cdots (\deg \cV_r)^{m_r}}^{1/(m_1 + \cdots + m_r - 1)},
    \]
    where
    \[
        C'_{k_1, \dots, k_r; m_1, \dots, m_r} = \left(\frac{d!}{\prod_{i=1}^{r} k_i!^{m_i}m_i!}\right)^{1/(m_1+\cdots+m_r-1)}.
    \] 
\end{enumerate}
\end{theorem}

Let us explain how various specializations of \cref{thm:main} correspond to earlier results.
\begin{enumerate}
    \item (Joints of lines) \cref{thm:joints-lines} corresponds to \cref{thm:main}\ref{part:const-no-mult} for $r= 1$, $k_1 = 1$, $m_1 = d$, and degree 1 varieties. In this case, the optimal constant $C_{1;d} = (d-1)!^{1/(d-1)}/d$ was determined previously in \cite{YZ} and matches the constant above.
    \item (Joints of flats) \cref{thm:flats} corresponds to \cref{thm:main}\ref{part:const-no-mult} for $r=1$, $k_1 = k$, $m_1 = m$, $d = km$, and degree 1 varieties.
    \item (Multijoints of lines) \cref{thm:mult-lines} corresponds to \cref{thm:main}\ref{part:const-no-mult} with $r = d$, $(k_i, m_i) = (1,1)$ for all $i$, and degree 1 varieties. This case was previously known~\cite{Zha20}. 
    The constant in this case was improved to $C_{1,\dots,1;1,\dots,1} = d!^{1/(d-1)}$ in \cite{YZ} (matching above) though it is likely not optimal. 
    Without consideration of constants, \cref{thm:mult-lines} also easily implies the setting allowing $m_i \ge 1$ by duplicating the sets of lines.
    \item (Multijoints of $k$-flats) \cref{thm:mult-flats} relaxes the $k_i = 1$ assumption above to arbitrary $k_i \ge 1$. Previously the only other known case is $(r;k_1, k_2; m_1,m_2) = (2;k,1; 1,d-k)$, i.e., a set of $k$-flats and a set of lines, where each joint is formed by one $k$-flat and $d-k$ lines, as proved independently by \cite{CI20} and \cite{YZ} (and stated without proof in \cite{Yang}). Even the ``next'' case of $(r;k_1, k_2; m_1, m_2) = (2;2,1;2,1)$ was previously unknown, corresponding to having a joint being formed by two flats and one line. Likewise, the case $(r;k_1, k_2, k_3; m_1, m_2, m_3) = (3; 2,1,1;1,1,1)$ allowing one set of flats and two different sets of lines was also previously unsolved.
    \item (Varieties) \cref{thm:joints-var,thm:mult-var} relax the degree 1 assumption, generalizing from flats to varieties. The only previously known case was for a single set of curves~\cite{KSS10,Qui09}, namely $r=1$ and $k_1 = 1$, as well as multiple sets of bounded degree curves in $\RR^n$~\cite{Ili15b}. \cref{thm:mult-var} is equivalent to \cref{thm:main}\ref{part:const-no-mult} (other than constants).
    \item (Multiplicities) Finally, adding in considerations of joint multiplicities, \cref{thm:car-lines} is equivalent to \cref{thm:main}\ref{part:const-mult} for lines, while \cref{thm:car-var} is equivalent to \cref{thm:main}\ref{part:const-mult} in general (other than constants).
    For a single set of lines, i.e., $(r,k_1,m_1) = (1,1,d)$, our result gives $C_{1;d} = 1$.
\end{enumerate}

While we know the optimal constant for joints of lines, our proof does not seem to give the optimal constant for flats or varieties. For $r=1$ we conjecture that the optimal constant in \cref{thm:main} is $C_{k,m} = (m!/m^m)^{1/(m-1)} N^{m/(m-1)}$ for all $k$ and $m$, agreeing with joints of lines ($k=1$). 
The first open case $(k,m) = (2,3)$ is stated below.

\begin{conjecture}
A set of $N$ planes in $\FF^6$ has at most $(\sqrt{2}/3 + o(1))N^{3/2}$ joints.
\end{conjecture}

\subsection{Outline}
We begin by motivating and describing, in \cref{sec:motivation}, the key new ideas in our method.
We then give, in \cref{sec:joints-of-planes}, the proof in the special case of joints of planes in $\RR^6$, which is representative of the general result.
To obtain the result in full generality, we use higher order directional derivatives with respect to local coordinates along a variety, as well as Hasse derivatives to deal with arbitrary fields, and they are both discussed in \cref{sec:derivatives}.
The complete proof of the main theorem then appears in \cref{sec:carbery}.

\section{Key ideas} \label{sec:motivation}

\subsection{Joints of lines}

We begin by recalling the proof of \cref{thm:joints-lines} on joints of lines in $\RR^3$ following \cite{KSS10,Qui09} (also see Guth's book \cite[Section 2.5]{Guth-book} for a nice exposition). The proof exposes two tools that are essential in nearly all applications of the polynomial method: \emph{parameter counting} and \emph{vanishing lemma}.

Let $\RR[x_1, \dots, x_d]_{\le n}$ denote the space of polynomials with degree at most $n$. Using that its dimension is $\binom{n+d}{d}$, we have the following simple yet extremely useful linear algebraic consequence.

\begin{lemma}[Parameter counting]
	Given a set of fewer than $\binom{n+3}{3}$ points in $\RR^3$, there exists a nonzero polynomial of degree at most $n$ that vanishes on all these points.
\end{lemma}

Given a set of $N$ lines forming $J$ joints in $\RR^n$, let $g$ be a nonzero polynomial of minimum degree that vanishes at all $J$ joints. By the parameter counting lemma, we have $\deg g \le C J^{1/3}$ for some constant $C > 0$.

The following elementary fact is key to the polynomial method.

\begin{lemma}[Vanishing lemma]
	If a degree $n$ polynomial vanishes at more than $n$ points on a line, then it vanishes on the whole line.
\end{lemma}

We claim that some line contains at most $C J^{1/3}$ joints. Suppose, for contradiction, that every line contains more than $CJ^{1/3}$ joints. 
Since $\deg g \le C J^{1/3}$, the vanishing lemma implies that $g$ vanishes on each of the $N$ lines. 
Since each joint is contained in three lines in spanning directions, the gradient $\nabla g$ vanishes at every joint.
Thus $\partial g/\partial x$, $\partial g/\partial y$, $\partial g/\partial z$ all vanish at every joint. At least one of these partial derivatives is a nonzero polynomial of degree smaller than that of $g$, thereby contradicting the minimal degree assumption on $g$.

Thus some line contains at most $C J^{1/3}$ joints. We can then remove this line and all its joints, and repeat the argument to find another line with at most $CJ^{1/3}$ joints. After we have removed all the lines, we have removed at most $CJ^{1/3} N$ joints, so $J \le CJ^{1/3} N$, and hence $J = O(N^{3/2})$. This completes the proof in the case of $\RR^3$. This proof also extends to $\FF^d$.

\subsection{Vanishing on planes}

How can we try to adapt the above proof to show that $N$ planes in $\RR^6$ form $O(N^{3/2})$ joints? The main obstacle is to generalize the vanishing lemma from lines to planes. The above proof would extend verbatim to joints of planes if the answer to the following question were yes.

\medskip

\noindent\emph{Attempt I.} Given distinct points $p_1,\ldots,p_{\binom{n+2}2}$ in the plane, if $g\in\RR[x,y]_{\leq n}$ satisfies the vanishing conditions $g(p_1)=0,\ldots,g(p_{\binom{n+2}2})=0$, does this imply that $g$ is identically zero?

\medskip

Of course, the answer to this question is no, since the vanishing locus of the polynomial on a plane could be a curve.
Clearly it is impossible to force a two-variable polynomial to vanish by forcing it to vanish at any finite number of points. Instead of asking for polynomials to vanish at the joints, we can ask them to vanish to high multiplicity at the joints.
This idea, known as the ``method of multiplicities''~\cite{DKSS13}, has been fruitful in the study of the joints problem  \cite{Zha20,YZ}, and it was also used to improve bounds on the finite field Kakeya problem~\cite{Dvir09,BW21}. 

\medskip

\noindent\emph{Attempt II.} Given a point $p_1$ in the plane, if $g\in\RR[x,y]_{\leq n}$ vanishes to order more than $n$ at $p_1$; equivalently, if $g$ satisfies the vanishing conditions $\frac{\partial^{i+j}g}{\partial x^i\partial y^j}(p_1)=0$ for all $0\leq i+j\leq n$, does this imply that $g$ is identically zero?

\medskip

The answer to this one is yes, and it shows how using derivatives creates a correct vanishing lemma. However, this vanishing lemma is completely useless for our application since we want to use this vanishing lemma somehow to bound the number of joints lying on a plane and this method ignores all of the joints but one on each plane.
Perhaps we can create a correct and useful vanishing lemma by combining the ideas of Attempts I and II.

\medskip

\noindent\emph{Attempt III.} Given distinct points $p_1,\ldots, p_m$ in the plane with $m\sim n^2/s^2$, if $g\in \RR[x,y]_{\leq n}$ vanishes to order at least $s$ at each point, does this imply that $g$ is identically zero?

\medskip

Unfortunately the answer is no again. Indeed $g(x,y) = y^s$ vanishes to order $s$ on the entire $x$-axis. 

We have $\dim \RR[x,y]_{\le n} = \binom{n+2}{2}$, less than the number of linear constraints on the coefficients of $g$ imposed by asking $g$ to vanish to order at least $s$ on $\Theta(n^2/s^2)$ given points (each such point gives $\binom{s+2}{2}$ constraints). 
This counterexample must imply that some of these linear constraints are linearly dependent.
Our proof strategy is to build a vanishing lemma using a linearly independent set of such constraints on the coefficients of $g$.

\begin{remark}
Another very natural strategy for extending the proof of joints of lines to planes is to consider, instead of a single polynomial that vanishes on all the joints, now a pair of polynomials that vanish on all the joints. For this approach to be useful, one would like the pair of polynomials, when restricted to each plane, to either be coprime or one of them to vanish. This seems like a difficult condition to satisfy and we suspect that it is not possible, at least if one wants the degrees of the polynomials to be small.

This problem appears to be related to the \emph{inverse B\'ezout problem}. 
Given a set of $N$ points in $\RR^2$, can one always find a pair of coprime polynomials $P, Q$ both vanishing on all $N$ points and $(\deg P)(\deg Q) = O(N)$?
The answer is no, by putting half of the $N$ points on a $\sqrt{N/2} \times \sqrt{N/2}$ grid and the other half on a line (this grid-and-line example shows up again in our discussion below). A partial converse to B\'ezout's theorem is known in 2-dimensions but open in higher dimensions (see Tao \cite{Tao-blog-bezout}).
\end{remark}

\subsection{Key idea I: collecting linearly independent vanishing conditions}

We define a \emph{vanishing condition} to be a single homogeneous linear constraint on the coefficients of a polynomial $g \in \RR[x,y]_{\le n}$ that arises from requiring some particular higher order directional derivative to vanish at some point. 
For example, for a two variable polynomial $g$, some examples of vanishing conditions are
(a) $g(2,4) = 0$,
(b) $\frac{\partial g}{\partial x}(2,1) = 0$,
and
(c) $\paren{\frac{\partial^2 g}{\partial x^2}-\frac{\partial^2 g}{\partial x\partial y}}(-1,2) = 0$.
For a positive integer $r$, an \emph{$r$-th order vanishing condition} on $g$ at $p$ is a vanishing condition of the form $Dg(p) = 0$ where $D$ is an $(r-1)$-th order derivative operator, i.e., a linear combination of $\partial^{r-1}/\partial^{r_1} x_1 \cdots \partial^{r_d}x_d$ for some $r_1+\cdots + r_d = r-1$. (We will not need mixed order vanishing conditions for joints of flats, but they will be needed for joints of varieties.)

For now, let us focus on a single plane and study vanishing conditions on $g \in \RR[x,y]_{\le n}$.
Vanishing conditions can be viewed as linear functionals on the vector space $\RR[x,y]_{\le n}$, though it will be helpful later to also keep track of the (derivative operator, point) pair $(D,p)$ that generates the vanishing condition $Dg(p) = 0$.

We now devise a procedure for selecting a basis of linear functionals on $\RR[x,y]_{\le n}$.

As a first attempt, we fix an arbitrary order on $\cP$, say $p_1, \dots, p_r$ and cycle through the points
(the vertical bars are a visual aid separating the epochs)
\[
p_1 \ p_2 \ \cdots p_r \mid
p_1 \ p_2 \ \cdots p_r \mid
p_1 \ p_2 \ \cdots p_r \mid \cdots.
\]
We cycle through the points in the above sequence and maintain a linearly independent set of vanishing conditions on $\RR[x,y]_{\le n}$, starting from an empty set of vanishing conditions.
The $r$-th time ($r = 1, 2, \dots$) that we see a point $p$, we add to our existing collection a maximal subset of $r$-th order vanishing conditions so that our collection of vanishing conditions always remains linearly independent as a set of linear functionals on $\RR[x,y]_{\le n}$. Eventually, the process terminates once we have collected a basis of $\binom{n+2}{2}$ linear functionals on $\RR[x,y]_{\le n}$.

Although there is some choice in the above process in deciding which vanishing conditions to add to our collection at each step, the \emph{number} of vanishing conditions added at each step does not depend on this choice. 
We would like to understand and control the number of vanishing conditions attached to each point as we run through the process. 
However, this does not seem easy. We do not know how to compute these numbers (for large $n$) even for an explicitly given set of points.

More importantly, the process does not always evenly assign the vanishing conditions across all the points. 
For example, suppose we have $\abs{\cP} = 2t^2$, with half of the points in $\cP$ forming an $t \times t$ grid (a high-degree part), and the other $t^2$ points all lying on a single generic line (a low-degree part).
As we run through the above process, we encounter significantly more linear dependencies among vanishing conditions at points on the line than on the grid.
For large $n$, at the end of the process, each point on the grid receives on the order of $t$ times as many vanishing conditions as each point on line.
This is an undesirable situation, since the process leads to an unequal distribution of vanishing conditions, effectively ``ignoring'' the points on the low-degree algebraic structure.

\subsection{Key idea II: handicaps and priority order}

To address the uneven distribution of vanishing conditions across points, we give the ``disadvantaged'' points a head start and cycle just among themselves many times before we cycle through the entire set of points.
For example, in the earlier grid-and-line example, if $p_1, \dots, p_{r/2}$ are points on the line and $p_{r/2+1}, \dots, p_r$ are points on the grid, then we give points on the line a head start, e.g., 
\[
p_1 \ p_2 \cdots p_{r/2} \mid
\cdots\mid
p_1 \ p_2  \cdots p_{r/2} \mid
p_1 \ p_2 \cdots p_r \mid 
p_1 \ p_2 \cdots p_r \mid \cdots.
\]
More generally, we give each point $p$ a \emph{handicap} $\alpha_p \in \ZZ$ corresponding to the number of rounds of head start.

For example, suppose there are five points labeled $a,b,c,d,e$ that we would cycle through in this order. 
Now we assign handicaps $0,1,3,0,-1$ to $a,b,c,d,e$ respectively. Then, for instance, $c$ starts in round $-3$ and $b$ starts in round $-1$. So we process the points in the following \emph{priority order}:
\[
c \mid
c  \mid
b \ c  \mid
a\ b\ c\ d\   \mid
a\ b\ c\ d\  e  \mid
a\ b\ c\ d\  e  \mid
\cdots.
\]
We now run the same vanishing condition collection process as earlier with this sequence of points.
The $r$-th time ($r = 1, 2, \dots$) that we see a point $p$, we append to our existing collection a maximal non-redundant set of $r$-th order vanishing conditions at $p$.

We would like to assign handicaps in a way so that all joints are treated equitably in the distribution of vanishing conditions (what this means precisely will be explained later). However, it appears to be a very difficult problem to determine how exactly the distribution of vanishing conditions depends on the handicaps. Intuitively, as in the grid-and-line example, we want to assign more handicap to points that are part of low-degree algebraic substructures, but it is far from obvious how to make this notion precise and useful.

\subsection{Key idea III: existence of a good handicap via compactness/smoothing}

Instead of explicitly assigning handicaps, we shall indirectly prove the existence of a good choice of handicaps via a compactness/smoothing argument. (Strictly speaking, we do not actually invoke compactness here since all our domains are finite, but we believe that compactness offers a helpful perspective as the argument here is a significant generalization of the earlier compactness argument giving tight bounds for joints of lines~\cite{YZ}.)

Fix a joints configuration. Let $n$ be large and consider the function
\begin{equation}\label{eq:handicap-partition-fn}
\text{handicaps $\alpha \in \ZZ^\cP$} \longrightarrow
\text{partitions of $\tbinom{n+2}{2}$ among $\cP$}
\end{equation}
where the partition records the final number of vanishing conditions assigned to each point.
While it appears to be difficult to compute this function explicitly, we can show that it has the following three properties.

\medskip

\noindent \emph{Bounded domain.} If one point has a much bigger handicap than another point, then the latter point gets assigned no vanishing conditions since the process would have finished before the first appearance of the latter point. Such a situation will never be desirable, so we only need to consider cases where the handicaps are all bounded (as a function of $n$).

\smallskip
    
\noindent \emph{Monotonicity.} 
Suppose we increase the handicap by one at a subset of points while holding others fixed.
Then the number of vanishing conditions assigned to this subset of points cannot decrease, and the number of vanishing conditions assigned to the other points cannot increase. 
Indeed, the points with the increased handicap now appear earlier in the priority order, and thus cannot receive fewer vanishing conditions than before the change. 

\smallskip
    
\noindent \emph{Lipschitz continuity.} A small change in the handicap assignments can only induce a small change in the number of vanishing conditions at each point. This property is intuitively reasonable, but it requires a proof.

\medskip

With these three properties, we can iteratively increase the handicaps at points that end up with too few constraints, so that we eventually balance out the distribution of constraints across all joints.

The eventual implicit assignment of handicaps across joints appears to somehow identify the ``algebraicity'' of each point in the configuration by assigning higher handicaps to points lying in lower-degree algebraic substructures. However, we do not know how to make this algebraicity intuition precise.

\begin{remark}
This idea of implicitly assigning handicaps came up in a simpler form previously in the work of Yu and Zhao~\cite{YZ} in determining the tight constant for the joints theorem of lines. 
There one does not have to consider any priority order or iterative process of adding constraints as we do here, though one does end up proving, via compactness, the existence of a handicap (though not called by that name) along with other parameters for controlling the order of vanishing at each joint.
\end{remark}

\subsection{Putting everything together: a new vanishing lemma}

Suppose we have a set $\cF$ of planes in $\RR^6$ forming joints $\cJ$.
For a choice of handicaps $\vec \alpha \in \ZZ^\cJ$, and a large integer $n$, we can run the above vanishing condition collection procedure separately on each plane (using handicaps $\vec\alpha$ restricted to points on the plane). 
On each plane $F \in \cF$, and at each joint $p$ on the plane $F$, the procedure attaches a set $\cD_{p,F} = \cD_{p,F}(\vec\alpha, n)$ of derivative operators.
Combining these vanishing conditions over all joints on $F$ then gives  a basis of linear functionals on the space of polynomials $g$ on $F$ of degree at most $n$, where each basis element is a vanishing condition of the form $Dg(p) = 0$ with $p \in \cF$ and $D \in \cD_{p,F}$ being a linear combination of higher order directional derivatives along $F$. With this data, we can now state our new vanishing lemma for joints of planes.

\medskip

\noindent\textbf{Vanishing lemma for joints of planes} (\cref{lem:vanishing}). \emph{With the above setup, if $g \in \RR[x_1, \dots, x_6]_{\le n}$ satisfies $D_1D_2D_3g(p) = 0$ whenever 
$D_i \in \cD_{p, F_i}$ are three derivative operators attached to three planes $F_1, F_2, F_3$ forming a joint $p \in \cJ$, then $g = 0$.}

\medskip 

Note that we are choosing a minimal set of derivative operators on each plane (as we chose a basis of linear functionals).
The vanishing lemma would be trivial if each $\cD_{p,F_i}$ were the full set of directional derivative operators at $p$ along $F$.

Also, our proof of the vanishing lemma only works if we build the vanishing conditions following the priority order---we would not be able to say much if the joints were processed in some other arbitrary manner.

By \textbf{parameter counting}, this new vanishing lemma implies the following inequality. Summing over joints $p$ formed by a triple of planes $F_1, F_2, F_3$, we have
\[
\sum_{(p, F_1, F_2, F_3)} \abs{\cD_{p,F_1}}\abs{\cD_{p,F_2}}\abs{\cD_{p,F_3}} \ge \dim \RR[x_1, \dots, x_6]_{\le n} = \binom{n+6}{6}.
\]
The left-hand side is the number of linear constraints on $g$ of the form $D_1D_2D_3g(p) = 0$ in the vanishing lemma. Indeed, if this inequality were not satisfied, by parameter counting there would be a non-zero polynomial $g$ of degree at most $d$ satisfying these vanishing conditions. However, the vanishing lemma implies that such a $g$ is identically zero, a contradiction.

Recall that all these quantities $\abs{\cD_{p,F}}$ depend on $n$ as well as the handicap $\vec\alpha$.
We can now apply a compactness/smoothing argument to choose a handicap $\vec\alpha$ that minimizes 
\[
\max_p \abs{\cD_{p,F_1}}\abs{\cD_{p,F_2}}\abs{\cD_{p,F_3}} 
-
\min_p \abs{\cD_{p,F_1}}\abs{\cD_{p,F_2}}\abs{\cD_{p,F_3}}.
\]
Using the three properties (bounded domain, monotonicity, Lipschitz continuity) of \cref{eq:handicap-partition-fn}, we can deduce that the above difference must be negligible, i.e., $o(n^6)$, since otherwise we can significantly reduce the above difference by increasing the handicap by 1 at a subset of points $p$ with small $\abs{\cD_{p,F_1}}\abs{\cD_{p,F_2}}\abs{\cD_{p,F_3}}$.

It follows that we can choose handicaps so that the product $\abs{\cD_{p,F_1}}\abs{\cD_{p,F_2}}\abs{\cD_{p,F_3}}$ is roughly constant across all $(p,F_1,F_2,F_3)$.
We also know that for each plane $F$, $\sum_{p \in F}\abs{\cD_{p,F}} = \dim \RR[x,y]_{\le n} = \binom{n+2}{2}$ since we have a basis of linear functionals on the space of polynomials on $F$ with degree at most $n$.
The conclusion $\abs{\cJ} = O(N^{3/2})$ then follows from a short calculation using the AM-GM inequality (see the end of \cref{sec:joints-of-planes}).

\medskip

In \cref{sec:joints-of-planes}, we flesh out these ideas to give a complete proof of joints of planes in $\RR^6$. In \cref{sec:derivatives} we discuss two further modifications to the above proof technique. To deal with varieties, we modify our notion of higher order directional derivatives. Geometrically we are taking derivatives with respect to local coordinates on the varieties. To deal with general fields other than the reals, we use Hasse derivatives.

\section{Joints of planes in $\RR^6$}\label{sec:joints-of-planes}

The purpose of this section is to prove that $N$ planes in $\RR^6$ have $O(N^{3/2})$ joints. This special case contains many of the key ideas that we introduce in this paper towards the full theorem.

Let $(\cJ, \cF)$ be a \emph{joints configuration} of planes in $\RR^6$, where $\cF$ is a finite set of planes and $\cJ$ is the set of joints formed by any three planes in $\cF$.
We abuse notation slightly to handle the case when more than three planes pass through $p \in \cJ$: in this case we arbitrarily choose three planes forming a joint at $p$, and only write ``$p \in F$'' (and say that ``$F$ contains $p$'', etc.) if $F$ is among the triple of planes chosen at $p$.

\subsection{Priority order and handicaps}
First, assign an arbitrary but fixed order (referred to as the \emph{preassigned order}) to the joints $\cJ$.

A \emph{handicap} $\vec{\alpha} = (\alpha_p)_{p \in \cJ} \in \ZZ^{\cJ}$ assigns an integer to each joint.
Given a handicap, the associated \emph{priority order} is a linear order on $\cJ\times \ZZ_{\ge 0}$ defined by setting $(p,r)\prec (p',r')$ 
\begin{itemize}
    \item if $r-\alpha_p<r'-\alpha_{p'}$, or
    \item if $r-\alpha_p=r'-\alpha_{p'}$ and $p$ comes before $p'$ in the preassigned order on $\cJ$.
\end{itemize}
The priority ordering corresponds to the description in the previous section. 
Note that in particular $(p,0) \prec (p,1) \prec (p,2) \prec \cdots$.
We write $\prec$ for the strict ordering, and $\preceq$ to allow equality.

\subsection{Derivatives and evaluations} \label{sec:deriv-eval}

Let $\RR[x_1, \dots, x_k]_{\le n}$ denote the space of polynomials of degree at most $n$ in $k$ variables.

Given a plane $F$ and a joint $p \in F$, let $\totalD_{p, F}^r$ denote the space of all $r$-th order derivative operators in directions along $F$, i.e., every element $D \in \totalD_{p,F}^r$ gives a linear map $g \mapsto D g$ sending $\RR[x_1 \dots, x_6] \to \RR[x_1, \dots, x_6]$ and $D$ is a linear combination of compositions of $r$ directional derivative operators along $F$. For example, if $F$ is the plane spanned by the first two coordinate directions, then $\totalD_{p, F}^r$ is the space spanned by the operators $\partial^{i+j}/\partial x_1^i \partial x_2^j$ ranging over all $i+j = r$. (The space $\totalD_{p, F}^r$ here does not actually depend on $p$, but we include $p$ in the notation with a view towards generalization from flats to varieties.)

Let $\totalB_{p, F}^r(n)$ denote the subspace of all linear functionals on $\RR[x_1, \dots, x_6]_{\le n}$ of the form $g \mapsto Dg(p)$ for some $D \in \totalD_{p,F}^r$ (i.e., an $r$-th order derivative along $F$ evaluated at $p$).
Then, for a fixed $p \in \cJ\cap F$, a polynomial $g \in \RR[x_1, \dots, x_6]_{\le n}$ lies in the common kernel of $\totalB_{p,F}^0(n) + \totalB_{p,F}^1(n) + \cdots + \totalB_{p,F}^{r-1}(n)$ if and only if the restriction of $g$ to the plane $F$ vanishes to order at least $r$ at $p$. (By common kernel we mean the intersection of the kernels of all linear functionals in this space.)

To emphasize the difference between $\totalB$ and $\totalD$, the elements of $\totalD_{p,F}^r$ are derivative operators sending polynomials to polynomials, whereas the elements of $\totalB_{p,F}^r(n)$ are linear functionals sending polynomials of degree up to $n$ to scalars. 
Perhaps a helpful mnemonic is that $D$ stands for ``differentiation'' while $B$ stands for ``basis'' (we will soon use a basis of the space of linear forms on polynomials up to degree $n$).

For a fixed $F \in \cF$, 
let us describe a process where we go through pairs $(p,r) \in (\cJ \cap F) \times \ZZ_{\ge 0}$ according to the priority order, and at each step we choose a 
\[
\prioB_{p, F}^r(\vec \alpha, n) \subset \totalB_{p, F}^r(n).
\]
We will drop the dependencies on $\vec\alpha$, $n$, and $F$ when there is no confusion, i.e., we write $\prioB_{p}^r \subset \totalB_{p}^r$ for the above inclusion. In addition, all unions and direct sums in the following paragraph are taken over $(p',r')\in(\cJ\cap F)\times \ZZ_{\geq 0}$. 

Suppose we are at the start of step $(p, r)$.
At this point, we have already chosen some $\prioB_{p'}^{r'} \subset \totalB_{p'}^{r'}$ for each $(p',r') \prec (p,r)$ so that the disjoint union $\bigcup_{(p',r') \prec (p,r)} \prioB_{p'}^{r'}$ is a basis for $\sum_{(p',r') \prec (p,r)} \totalB_{p'}^{r'}$.
Now consider expanding this space to $\sum_{(p',r') \preceq (p,r)} \totalB_{p'}^{r'}$ by adding in all the $r$-th order derivative evaluations at $p$ along $F$. We desire to expand the basis accordingly. 
As such, we choose a set $\prioB_{p}^r \subset \totalB_p^r$ so that the disjoint union $\bigcup_{(p',r') \preceq (p,r)} \prioB_{p'}^{r'}$ becomes a basis of $\sum_{(p',r') \preceq (p,r)} \totalB_{p'}^{r'}$.
Note that while we have some choice about which elements of $\totalB_p^r$ to include as new basis elements, the size of $\prioB_{p}^r$ does not depend on any choice, and is only a function $n$ and the priority order. We will provide a more direct formula for $\abs{\prioB_p^r}$ shortly.

Since each element of $\totalB_{p,F}^r(n)$ can be written as $g \mapsto Dg(p)$ for some $D \in \totalD_{p,F}^r$, we can choose
\[
\prioD_{p, F}^r(\vec \alpha, n) \subset \totalD_{p, F}^r
\]
with the same size as $\prioB_{p,F}^r(\vec\alpha,n)$ so that 
\[
\prioB_{p, F}^r(\vec \alpha, n) = \{g \mapsto Dg(p) : D \in \prioD_{p, F}^r(\vec \alpha, n)\}.
\]

We write
\[
\prioB_{p, F}(\vec \alpha, n) := 
\bigcup_{r \ge 0} \prioB_{p, F}^r(\vec \alpha, n)
\quad\text{and}\quad
\prioD_{p, F}(\vec \alpha, n) := \bigcup_{r \ge 0} \prioD_{p,F}^r(\vec \alpha, n).
\]

As we range over all joints $p$ on $F$, the sets $\prioB_{p, F}(\vec \alpha, n)$ combine to form a basis of the space of linear forms on polynomials of degree at most $n$ on $F$.
Thus
\begin{equation}\label{eq:summation}
\sum_{p\in \cJ\cap F}\abs{\prioB_{p, F}(\vec{\alpha},n)}
=
\dim \RR[x,y]_{\le n}
=
\binom{n+2}{2}.	
\end{equation}

We may omit the parenthetical $\vec\alpha$ and $n$ in our notation when these parameters do not change and the context is clear.  Some of the arguments below will involve comparing different values of $\vec\alpha$ and $n$, in which case we will state the dependencies explicitly.
We may also omit $F$ when we are not considering other planes.

\subsection{Polynomials with given vanishing orders} \label{sec:poly-vanish}
In this and the next subsection, we focus our attention on a single fixed plane $F \cong \RR^2$. 
Fix a finite set of points $\cP \subset F$ (which we will later take to be the joints on $F$).
Given a vector $\vec v = (v_p)_{p \in \cP} \in \ZZ_{\ge 0}^{\cP}$, let 
\[
\totalT(\vec v, n)  = \{ g \in \RR[x,y]_{\le n} : g \text{ vanishes to order} \ge v_p \text{ at each } p\in\cP\}
\]
(i.e., the partial derivatives satisfy $\frac{\partial^{i+j}g}{\partial x^i\partial y^j}(p)=0$ for all $i+j<v_p$). We would like to understand how the dimension of $\totalT(\vec v, n)$ changes with $\vec v$ and $n$. We are particularly interested in the following quantity, which we will shortly relate below in \cref{eq:B-T} to $\sabs{\cB_{p,F}^r(\vec\alpha,n)}$: for $p \in \cP$, set
\[
b_p(\vec v, n) := \codim_{\totalT(\vec v, n)} \totalT(\vec v + \vec e_p, n)
= \dim \totalT(\vec v, n) - \dim \totalT(\vec v + \vec e_p, n).
\]
Here, given a pair of subspaces $W \le U$, we write $\codim_U W$ for the relative codimension of $W$ in $U$. Also $\vec e_p \in \ZZ^\cP$ is the vector with $1$ at $p$ and $0$ elsewhere.
Note, for each $p \in \cP$, the space $\totalT(\vec v + \vec e_p, n)$ is the nullspace of the map on $\totalT(\vec v, n)$ that sends every polynomial $g$ to all its $v_p$-th order derivatives evaluated at $p$, and thus $b_p(\vec v, n)$ is the rank of this map.

The following basic fact will be useful:
\begin{equation}
\label{eq:subspace-inequality}
\text{for subspaces }U,W \le V,\text{ we have }\codim_V W \ge \codim_U (W\cap U).
\end{equation}

\begin{lemma}[Bounded domain] \label{lem:Tzero}
If $\vec v \in \ZZ_{\ge 0}^\cP$ has $v_p > n$ for some $p \in \cP$, then $\dim \totalT(\vec v, n) = 0$. 
\end{lemma}

\begin{proof}
This is the statement that no nonzero polynomial of degree at most $n$ can vanish to order more than $n$ at some point.
\end{proof}

\begin{lemma}[Monotonicity] \label{lem:Tmono}
Let $p \in \cP$. Suppose $\vec v^{(1)}, \vec v^{(2)} \in \ZZ_{\ge 0}^\cP$ satisfy $\vec v^{(1)} \ge \vec v^{(2)}$ coordinatewise and with equality at $p$.
Then $b_p(\vec v^{(1)}, n) \le b_p(\vec v^{(2)},n)$ for all $n$.
\end{lemma}

\begin{proof}
Earlier we saw that for each $i=1,2$, $b_p(\vec v^{(i)},n)$ is the rank of the map on $\totalT(\vec v^{(i)}, n)$ that sends each polynomial to all its $v_p^{(1)} = v_p^{(2)}$-th order derivatives evaluated at $p$.
Since $\vec v^{(1)} \ge \vec v^{(2)}$ coordinatewise, $\totalT(\vec v^{(1)}, n)$ is a subspace of $\totalT(\vec v^{(2)}, n)$, which implies the inequality $b_p(\vec v^{(1)}, n) \le b_p(\vec v^{(2)},n)$ on the rank of a map when restricted to a subspace.
\end{proof}

The next two lemmas together will lead to the Lipschitz continuity property of $b_p(\vec v, n)$ as a function of $\vec v$.

\begin{lemma} \label{lem:Tlip1}
Let $p,q\in\cP$ be distinct points. 
Then for every $\vec v \in \ZZ_{\ge 0}^\cP$ and nonnegative integer $n$, one has
$b_p(\vec v + \vec e_q, n) \ge b_p(\vec v, n-1)$.
\end{lemma}

\begin{proof}
Let $f$ be an arbitrary linear polynomial that vanishes at $q$ but at no other point of $\cP$ (such $f$ clearly exists if the underlying field $\FF$ is large enough; if not, we replace $\FF$ by a field extension, which would not affect $b_p(\vec v, n)$ as it is a rank-type quantity). 
We have
\begin{align*}
b_p(\vec v + \vec e_q, n) 
&= \codim_{\totalT(\vec v + \vec e_q, n)} \totalT(\vec v + \vec e_p +  \vec e_q, n)
\\
&\ge \codim_{f \cdot \totalT(\vec v , n-1)} f \cdot \totalT(\vec v + \vec e_p, n-1)
\\
&= \codim_{\totalT(\vec v , n-1)} \totalT(\vec v + \vec e_p, n-1)
\\
&= b_p(\vec v, n-1). 
\end{align*}
The inequality step follows from \cref{eq:subspace-inequality}, observing that restricting $\totalT(\vec v + \vec e_q, n)$ and $\totalT(\vec v + \vec e_p +  \vec e_q, n)$ to polynomials divisible by $f$ yields $f \cdot \totalT(\vec v , n-1)$ and $f \cdot \totalT(\vec v + \vec e_p, n-1)$ respectively.
\end{proof}

\begin{lemma} \label{lem:Tlip2}
Let $p \in \cP$. Suppose $\vec v^{(0)}, \vec v^{(1)}, \dots \in \ZZ^\cP$ are such that $\vec v^{(0)} \le  \vec v^{(1)} \le \cdots$ coordinate-wise and strictly increasing at the coordinate indexed by $p$.
Then
\[
\sum_{r \ge 0} b_p(\vec v^{(r)}, n) - \sum_{r\ge 0} b_p(\vec v^{(r)}, n-1)  \le n+1.
\]
\end{lemma}

\begin{proof}
For each $r \ge 0$, we have
\begin{align*}
b_p(\vec v^{(r)}, n) - b_p(\vec v^{(r)}, n-1)
&= \codim_{\totalT(\vec v^{(r)}, n)} \totalT(\vec v^{(r)} + \vec e_p, n) - \codim_{\totalT(\vec v^{(r)}, n-1)} \totalT(\vec v^{(r)} + \vec e_p, n-1)
\\
&= \codim_{\totalT(\vec v^{(r)}, n)} \totalT(\vec v^{(r)}, n-1) - \codim_{\totalT(\vec v^{(r)} + \vec e_p, n)} \totalT(\vec v^{(r)} + \vec e_p, n-1).
\end{align*}
We have $\codim_{\totalT(\vec v^{(r)} + \vec e_p, n)} \totalT(\vec v^{(r)} + \vec e_p, n-1) \ge \codim_{\totalT(\vec v^{(r+1)} , n)} \totalT(\vec v^{(r+1)}, n-1)$ by \cref{eq:subspace-inequality} since $\vec v^{(r)} + \vec e_p \le \vec v^{(r+1)}$ coordinatewise. Summing over all $r \ge 1$, we obtain
\begin{align*}
\sum_{r \ge 0} b_p(\vec v^{(r)}, n) - \sum_{r \ge 0} b_p(\vec v^{(r)}, n-1)  
&\le \codim_{\totalT(\vec v^{(0)}, n)} \totalT(\vec v^{(0)}, n-1) 
\\
&\le \codim_{\totalT(\vec 0, n)} \totalT(\vec 0, n-1) \\
&= n+1. \qedhere
\end{align*}
\end{proof}

\begin{lemma}[Lipschitz continuity] \label{lem:Tlip}
Let $p,q\in\cP$ be distinct points. Suppose $\vec v^{(0)}, \vec v^{(1)}, \dots \in \ZZ^\cP$ are such that $\vec v^{(0)} \le  \vec v^{(1)} \le \cdots$ coordinate-wise and strictly increasing at the coordinate indexed by $p$. Then
\[
0\le \sum_{r \ge 0} b_p(\vec v^{(r)}, n)
- \sum_{r \ge 0} b_p(\vec v^{(r)} + \vec e_q, n) 
\le n+1.
\]
\end{lemma}

\begin{proof}
Combine \cref{lem:Tmono,lem:Tlip1,lem:Tlip2}.
\end{proof}

\subsection{How the number of vanishing conditions varies with the handicap} \label{sec:vanish-handicap}
As in the previous subsection, let us continue to focus our attention on a set of points $\cP$ on a fixed plane $F\cong\RR^2$ (which we will drop from our notation temporarily).

Given a handicap $\vec\alpha \in \ZZ^\cP$ (restricted to this plane),  we define the vector $\vec{v}^{p,r}(\vec{\alpha})$ as follows. 
It assigns to coordinate $p' \in \cP$ the smallest nonnegative integer $r'$ such that $(p,r) \preceq(p',r')$. Equivalently, the value of $\vec{v}^{p,r}(\vec{\alpha})$ at $p'$ is given by
\begin{equation} \label{eq:v-alpha}
v^{p,r}_{p'}(\vec{\alpha}) 
= \begin{cases}
\max\{r - \alpha_p + \alpha_{p'} + 1,0\} & \text{if $p'$ comes strictly before $p$ in the preassigned order,}
\\
\max\{r - \alpha_p + \alpha_{p'} ,0\} & \text{otherwise.}
\end{cases}
\end{equation}
In other words, $\vec v^{p,r}(\vec \alpha)$ collects the desired vanishing orders at each joint on $F$ at the stage right before we hit $(p,r)$ in the priority order. 

Define $\prioB_p^r(\vec\alpha,n)$ and $\prioB_p(\vec\alpha,n)$ as in \cref{sec:deriv-eval} restricted to this plane.
Recall that for every $(p,r) \in \cP \times \ZZ_{\ge 0}$,
the disjoint union $\bigcup_{(p',r') \prec (p,r)} \prioB_{p'}^{r'}$ is basis of $\sum_{(p',r') \prec (p,r)} \totalB_{p'}^{r'}$. 
Then a polynomial $g \in \RR[x_1, \dots, x_6]_{\le n}$ lies in the common kernel of $\bigcup_{(p',r') \prec (p,r)} \prioB_{p'}^{r'}$ if and only if the restriction of $g$ to the plane $F$ vanishes to order at least $v_q^{p,r}(\vec\alpha)$ for every $q \in \cP$. Since adding $\prioB_p^r$ makes this set a basis for $\sum_{(p',r') \preceq (p,r)} \totalB_{p'}^{r'}$, its size is the number of non-redundant constraints that we need to add to increase the order of vanishing at $p$ by $1$. Thus
\begin{equation}\label{eq:B-T}
\abs{\prioB_p^r(\vec \alpha, n)}
= b_p(\vec v, n)
= \codim_{\totalT(\vec v, n)} \totalT(\vec v + \vec e_p, n)
\quad \text{ with } 
\vec v = \vec v^{p,r}(\vec \alpha).
\end{equation}

The observations in the previous section then imply the following.

\begin{lemma}[Bounded domain] \label{lem:zero}
Let $n \ge 0$ and $\vec\alpha \in \ZZ^\cP$. Let $p,q \in \cP$.
If $\alpha_p < \alpha_{q}-n$,
then $\abs{\prioB_p(\vec\alpha, n)}=0$. 
\end{lemma}
\begin{proof}
For each $r \ge 0$, the value of $\vec v = \vec v^{p,r}(\vec\alpha)$ at $q$ is greater than $n$, so $\dim \totalT(\vec v, n) = 0$ by \cref{lem:Tzero}. 
Hence $\abs{\prioB_p^r(\vec \alpha, n)} = b_p(\vec v, n) = 0$.
\end{proof}

\begin{lemma}[Monotonicity] \label{lem:mono}
Let $n$ be a positive integer and $\vec\alpha^{(1)},\vec\alpha^{(2)} \in \ZZ^\cP$ be two handicaps. 
Suppose $p \in \cP$ satisfies 
$\alpha^{(1)}_p-\alpha^{(1)}_{p'}\le \alpha^{(2)}_p-\alpha^{(2)}_{p'}$ for all $p' \in \cP$.
Then $\abs{\prioB_p(\vec\alpha^{(1)},n)} \le \abs{\prioB_p(\vec\alpha^{(2)},n)}$.
\end{lemma}

\begin{proof}
For each $i=1,2$, let $\vec v^{(i)} = \vec v^{(p,r)}(\vec \alpha^{(i)})$.
From \cref{eq:v-alpha} we see that $\vec v^{(1)} \ge \vec v^{(2)}$ coordinatewise and with equality at $p$. Then \cref{lem:Tmono} gives $b_p(\vec v^{(1)}, n) \le b_p(\vec v^{(2)},n)$, and \cref{eq:B-T} gives the claim.
\end{proof}

\begin{lemma}[Lipschitz continuity]\label{lem:lip}
Let $p \in \cP$ and $\vec\alpha^{(1)},\vec\alpha^{(2)} \in \ZZ^\cP$.
Then
\[
\abs{\sabs{\prioB_{p}(\vec\alpha^{(1)},n)}-\sabs{\prioB_{p}(\vec\alpha^{(2)},n)}}
\le 
(n+1)\sum_{p'\in \cP }\abs{(\alpha^{(1)}_{p'}-\alpha^{(1)}_{p})-(\alpha^{(2)}_{p'}-\alpha^{(2)}_{p})}.
\]
\end{lemma}

\begin{proof}
	Shifting all handicaps by the same constant does not change the priority order and thus also does not change $|\prioB_p|$. Since the right-hand side of the above inequality is also invariant under translation we may assume that $\alpha_p^{(1)}=\alpha_p^{(2)}=0$.
	
	Starting with $\vec \alpha = \vec \alpha^{(1)}$, we can perform a sequence of changes where at each step we change the value of the handicap $\vec \alpha$ at some $p'\ne p$ by exactly 1, so that the vector $(\alpha_{p'})_{p' \in \cP}$ ends up being equal to $(\alpha_{p'}^{(2)})$ after exactly $\sum_{p'\in\cP}\abs{\alpha^{(1)}_{p'}-\alpha^{(2)}_{p'}}$ moves.
	So it suffices to prove the inequality for each step in the process, i.e., showing that for every $\vec \alpha \in \ZZ^\cP$ and $q \ne p$,
	\[
		0 \le \sabs{\prioB_{p}(\vec\alpha,n)} - \sabs{\prioB_{p}(\vec\alpha + \vec e_q,n)} \le n+1.
	\]
	The first inequality follows from \cref{lem:mono}.
	For the second inequality, by $\abs{\prioB_p(\vec \alpha, n)} = \sum_{r \ge 0} \abs{\prioB_p^r(\vec \alpha, n)}$ and \cref{eq:B-T}, it suffices to prove
	\[
		\sum_{r \ge 0} b_p(\vec v^{p,r}(\vec \alpha), n) - 	\sum_{r \ge 0} b_p(\vec v^{p,r}(\vec \alpha + \vec e_q), n) \le n+1.
	\]
	From \cref{eq:v-alpha}, we see that there is some $r_0$ so that $\vec v^{p,r}(\vec \alpha + \vec e_q) = \vec v^{p,r}(\vec \alpha)$ for all $r < r_0$ and $\vec v^{p,r}(\vec \alpha + \vec e_q) = \vec v^{p,r}(\vec \alpha) + \vec e_q$ for all $r \ge r_0$. Restricting the sum to $r \ge r_0$ (the earlier terms cancel), we obtain the desired inequality by \cref{lem:Tlip}.
\end{proof}

\subsection{Vanishing lemma}
Now we start considering the interactions between different planes  at the joints.
The next statement is a vanishing lemma that is tailored to this joints problem.
We omit the dependence on the handicap $\vec \alpha$ and the degree $n$ from the notation since we are keeping them fixed in this subsection.
Recall from the beginning of the section that, at each joint, we arbitrarily chose three planes that form this joint.
Note that this vanishing lemma is the only place in the proof where we use the hypothesis that the three planes that form a joint do not all lie in some hyperplane.

\begin{lemma}\label{lem:vanishing}
	Let $(\cJ, \cF)$ be a joints configuration of planes in $\RR^6$.
	Given a handicap $\vec \alpha \in \ZZ^\cJ$ and its associated priority order, and a positive integer $n$, 
	choose $\prioD_{p,F}$ as earlier.
	
	Then for every nonzero polynomial $g \in \RR[x_1, \dots, x_6]$ of degree at most $n$, one has 
	\[
	D_1D_2D_3 g(p) \ne 0
	\]
	for some joint $p \in \cJ$ formed by $F_1, F_2, F_3 \in \cF$, and some $D_i \in \prioD_{p, F_i}$ for each $i=1,2,3$.
\end{lemma}

\begin{proof}
	Suppose, on the contrary, that there were some nonzero $g \in \RR[x_1, \dots, x_6]_{\le n}$ such that $D_1 D_2 D_3 g(p) = 0$ for every $p \in \cJ$, with $F_1, F_2, F_3 \in \cF$ being the three planes passing through $p$, and every $D_i \in \prioD_{p, F_i}$ for each $i=1,2,3$.
	
	Choose $p \in \cJ$ to minimize $(p, v_p(g))$ under $\prec$, where $v_p(g)$ is the order of vanishing of $g$ at $p$.
	
	Recall that $\totalD_{p,F}^r$ is the space of $r$-th order derivative operators at $p$ along $F$.
	Since $g$ vanishes to order exactly $v_p(g)$ at $p$ and the planes $F_1,F_2,F_3$ do not all line in one hyperplane, there exist $D_1 \in \totalD_{p, F_1}^{r_1}$, $D_2 \in \totalD_{p, F_2}^{r_2}$, $D_3 \in \totalD_{p, F_1}^{r_3}$ with $D_1D_2D_3 g (p) \ne 0$ and $r_1 + r_2 + r_3 = v_p(g)$. 
	Among all choices of $D_1, D_2, D_3$ (including choices of $r_1, r_2, r_3$), choose ones so that $|\{ i \in [3] : D_i \in \prioD_{p, F_i}\}|$ is maximized.	
	By the assumption at the beginning of the proof, one must have $D_i \notin \prioD_{p, F_i}$ for some $i$. Relabeling if necessary, assume that $D_1 \notin \prioD_{p, F_1}$. 
	
	Suppose $p' \in F_1 \cap \cJ$ and $r' \in \ZZ_{\ge 0}$ satisfy $(p', r') \prec (p, r_1)$. 
	We get $(p', r' + r_2 + r_3) \prec (p, r_1 + r_2 + r_3) = (p, v_p(g))$. 
	By the choice of $p$, we have $(p, v_p(g)) \preceq (p', v_{p'}(g))$.
	Thus $(p', r' + r_2 + r_3) \prec (p', v_{p'}(g))$, and hence $r' + r_2 + r_3 < v_{p'}(g)$.
	If follows that $D D_2 D_3 g (p') = 0$ for all $D \in \totalD_{p', F_1}^{r'}$ by the definition of vanishing order.
	
	From the above paragraph we deduce that $D_2D_3 g$ lies in the common kernel of $\totalB_{p', F_1}^{r'}$ ranging over all $(p',r') \in (F_1 \cap \cJ) \times \ZZ_{\ge 0}$ with $(p',r') \prec (p,r_1)$.
	Since $D_1D_2D_3 g(p) \ne 0$, we deduce that $D_2D_3 g$ does not lie in the common kernel of $\prioB_{p,F_1}^{r_1}$, i.e., there is some $D \in \prioD_{p, F_1}^{r_1}$ with $DD_2D_3 g(p) \ne 0$. But this $D$ contradicts the earlier assumption that the choice of $(D_1, D_2, D_3)$ maximizes $|\{ i : D_i \in \prioD_{p, F_i}\}|$.
\end{proof}

The next inequality uses parameter counting.

\begin{lemma}
\label{lem:param-ineq}
Assume the same setup as \cref{lem:vanishing}. We have
\[\sum_{p\in \cJ}\prod_{F\ni p}\abs{\prioD_{p, F}(\vec{\alpha},n)}\geq \binom{n+6}{6}.\]
\end{lemma}
\begin{proof}
Denote the left hand side by $A$ and right hand side by $B$. Consider the constraints on $g\in \RR[x_1,\ldots,x_6]_{\leq n}$ where for all $p\in \cJ$ formed by the planes $F_1,F_2,F_3\in\cF$, we require
\[D_1D_2D_3g(p)=0\quad\forall D_i\in \prioD_{p,F_i}, i=1,2,3.\]
This requirement is asking $A$ linear functionals on $\RR[x_1,\ldots,x_6]_{\leq n}$, which has dimension $B$, to vanish at $g$. Hence, if $A<B$, then there exists a nonzero polynomial $g$ in $\RR[x_1,\ldots,x_6]_{\leq n}$ that satisfies all the conditions, which would contradict \cref{lem:vanishing}.
\end{proof}
\subsection{Choosing the handicaps}
We say that a joints configuration $(\cJ,\cF)$ is \emph{connected} if the following graph is connected: the vertex set is $\cJ$, with two joints adjacent if there is some plane in $\cF$ containing both joints.

\begin{lemma}\label{lem:handicaps}
Let $(\cJ,\cF)$ be any connected joints configuration, and let $n$ be some positive integer. Then there exists a choice of handicap $\vec\alpha \in \ZZ^\cJ$ such that
\[\max_{p\in \cJ}\prod_{F\ni p}\frac{\abs{\prioD_{p, F}(\vec{\alpha},n)}}{\binom{n+2}{2}}-\min_{p\in \cJ}\prod_{F\ni p}\frac{\abs{\prioD_{p, F}(\vec{\alpha},n)}}{\binom{n+2}{2}}\leq \frac{C}{n}\]
for some constant $C$ that only depends on $(\cJ,\cF)$ but not $n$.
\end{lemma}
\begin{proof}
Fix $n$ throughout the proof. Denote
\[W_p(\vec{\alpha}) = \prod_{F\ni p}\frac{\abs{\prioD_{p, F}(\vec{\alpha},n)}}{\binom{n+2}{2}}\]
for all $p\in \cJ$. The $\alpha_p$ are arbitrary integers. However, note that shifting all $\alpha_p$ by the same constant does not affect the priority order and thus does not affect $W_p(\alpha)$. Furthermore, by \cref{lem:zero}, if two handicaps differ by more than $n$ at two points on the same plane, then $W_p(\vec\alpha)=0$. Therefore, there are only finitely many possibilities for the vector $(W_p(\vec{\alpha}) : p\in \cJ)$. Among those possibilities, choose the one so that after sorting $W_p(\vec{\alpha})$ in descending order, this vector is least in lexicographical order over all such possible vectors. Suppose that the sorted result is
\[W_{p_1}(\vec{\alpha})\geq W_{p_2}(\vec{\alpha})\geq \cdots\geq W_{p_{\abs{\cJ}}}(\vec{\alpha}). \]
We will show that $W_{p_i}(\vec{\alpha})-W_{p_{i+1}}(\vec{\alpha})\leq C'/n$ for some constant $C'$ to be determined. This will imply the desired statement.

Suppose for the sake of contradiction that the above claim does not hold. Let $t$ be the least positive integer such that $W_{p_t}(\vec{\alpha})-W_{p_{t+1}}(\vec{\alpha})>C'/n$. Then let $\vec{v}=\vec{e}_{p_1}+\cdots+\vec{e}_{p_t}$ and let $\vec{\alpha}'=\vec{\alpha}-\vec{v}$ be a new handicap. We will consider the difference between $W_p(\vec{\alpha})$ and $W_p(\vec{\alpha}')$. By \cref{lem:lip}, 
\[
\abs{\sabs{\prioD_{p, F}(\vec{\alpha},n)} - \sabs{\prioD_{p, F}(\vec{\alpha}',n)}}
\le
\abs{\cJ}(n+1) 
\le \frac{2\abs{\cJ}}{n} \binom{n+2}{2}
\]
for each joint $p$ on each plane $F$. We have $\abs{\prioD_{p, F}(\vec{\alpha},n)} \le \binom{n+2}{2}$ by \cref{eq:summation}. We use the following telescoping inequality. For $x_1,x_2,x_3,y_1,y_2,y_3\in[0,1]$, \[|x_1x_2x_3-y_1y_2y_3|\leq |x_1-y_1|x_2x_3+|x_2-y_2|y_1x_3+|x_3-y_3|y_1y_2\leq 3\max_i|x_i-y_i|.\] Thus
\[
\abs{W_p(\vec{\alpha}')-W_p(\vec{\alpha})}
\leq \frac{6|\cJ|}{n}=\frac{C'}{2n}
\]
where we choose $C'=12|\cJ|$.

By the monotonicity established in \cref{lem:mono}, we know that $W_{p_i}(\vec{\alpha}')\leq W_{p_i}(\vec{\alpha})$ for $i\leq t$, and $W_{p_i}(\vec{\alpha}')\geq W_{p_i}(\vec{\alpha})$ for $i>t$. 
By \cref{eq:summation}, we know that if  $W_p(\vec{\alpha}')\neq W_p(\vec{\alpha})$ for some $p$, then there exists $i\leq t$ such that $W_{p_i}(\vec{\alpha}')< W_{p_i}(\vec{\alpha})$.
However, since the difference between $W_p(\vec{\alpha})$ and $W_p(\vec{\alpha}')$ is at most $C'/2n$, and $W_{p_t}(\vec{\alpha})-W_{p_{t+1}}(\vec{\alpha})>C'/n$, we know that $W_{p_1}(\vec{\alpha}'),\ldots,W_{p_t}(\vec{\alpha}')$ are still the $t$ largest values among $(W_p(\vec{\alpha}'))_{p\in \cJ}$. 
This shows that $\vec{\alpha}'$ gives a strictly lower lexicographical order of $(W_p(\vec{\alpha}'))_{p\in\cJ}$, which is a contradiction.

Hence $W_p(\vec{\alpha})=W_p(\vec{\alpha}') = W_p(\vec{\alpha}-\vec{v})$ for all $p\in \cJ$. By the same argument, we know that $W_p(\vec{\alpha})=W_p(\vec{\alpha}') = W_p(\vec{\alpha}-c\vec{v})$ holds for $p\in\cJ$ for any positive integer $c$. By connectedness,  we can find some $i\leq t<j$ such that $p_i$ and $p_j$ are on the same plane. As a consequence, if $c$ is chosen sufficiently large such that $\alpha_{p_i}-c< \alpha_{p_j}-n$, this implies that $W_{p_i}(\vec{\alpha}-c\vec{v})=0$. By our ordering this implies that $W_{p_{i'}}(\vec{\alpha}-c\vec{v})=0$ for all $i'\geq i$. In particular, $W_{p_t}(\vec{\alpha})=W_{p_{t+1}}(\vec{\alpha})=0$, contradicting our earlier assumption that $W_{p_t}(\vec{\alpha})-W_{p_{t+1}}(\vec{\alpha})>C'/n$.
\end{proof}

We are now ready to prove the joints theorem for a set of planes in $\RR^6$.

\begin{proof}[Proof that $N$ planes in $\RR^6$ have $\sqrt{10/3}N^{3/2}$ joints]
Assume first that the joints configuration is connected. Let $n$ be some large positive integer. In this proof we will use $O$-notation to suppress constants that can depend on $(\cJ,\cF)$ arbitrarily as long as they are independent of $n$. Choose $\vec{\alpha}$ according to \cref{lem:handicaps}. Then there exists $W$ such that
\[\left|\prod_{F\ni p}\frac{|\prioD_{p, F}|}{\binom{n+2}{2}}-W\right|\leq \frac{C}{n}\]
for all $p\in \cJ$. By \cref{lem:param-ineq}, we have
\[|\cJ|W\binom{n+2}{2}^{3} \ge \binom{n+6}{6} - O(n^5).\]
Therefore
\[W\geq \frac{8}{6!\cdot |\cJ|} - O(n^{-1}).\]
So there is some constant $c>0$ (depending on $\cJ$ but not on $n$) so that $W \in [c,1]$ for all sufficiently large $n$.
For each $p \in \cJ$, by a Taylor series approximation,
\[
W^{1/3} = \left(\prod_{F\ni p}\frac{|\prioD_{p,F}|}{\binom{n+2}{2}}\right)^{1/3} + O(n^{-1}).
\]
Hence (in the summations, $p$ ranges over joints and $F$ ranges over planes in $\cF$),
\begin{align*}
    3|\cJ|\left(\frac{W}{N^3}\right)^{1/3}
    &= 3\sum_{p}\left(\prod_{F\ni p}\frac{|\prioD_{p, F}|}{N\binom{n+2}{2}}\right)^{1/3}+O(n^{-1})\\
    &\leq \sum_{p}\sum_{F\ni p}\frac{|\prioD_{p, F}|}{N\binom{n+2}{2}} + O(n^{-1})&&\text{\footnotesize [by AM-GM]}\\
    &= \sum_{F}\sum_{p\in F}\frac{|\prioD_{p, F}|}{N\binom{n+2}{2}}+O(n^{-1})\\
    &=\sum_{F}\frac{1}{N}+O(n^{-1})&&\text{[by \footnotesize \cref{eq:summation}]}\\
    &=1+O(n^{-1}).
\end{align*}
Thus
\[W\leq \frac{N^3}{27|\cJ|^3} + O(n^{-1}).\]
By comparing the leading term in the upper bound and the lower bound of $W$, i.e., letting $n$ go to infinity, we get that
\[\frac{8}{6!\cdot |\cJ|}\leq \frac{N^3}{27|\cJ|^3},\]
and by rearranging we get that
\[|\cJ|\leq \sqrt{\frac{10}{3}}N^{3/2}.\]
The above argument proves the result for connected joints configurations. In general, decompose the joints configuration $(\cJ,\cF)$ into connected components (in the sense of the associated graph) $(\cJ_1,\cF_1)$, \dots, $(\cJ_k,\cF_k)$. Denote $N_i=\abs{\cF_i}.$ Then
\[|\cJ|=\sum_{i=1}^{k}|\cJ_i|\leq \sqrt{\frac{10}{3}}\sum_{i=1}^{k}N_i^{3/2}\leq \sqrt{\frac{10}{3}}N^{3/2}. \qedhere\]
\end{proof}

\begin{remark}
The arguments here generalize straightforwardly to joints of flats in arbitrary dimensions.	
\end{remark}

\section{Derivatives along varieties}\label{sec:derivatives}

In this section we discuss how to generalize the argument in \cref{sec:joints-of-planes} to varieties in $\FF^d$. There are two issues that we need to address.
The first is to define appropriate higher order directional derivatives along varieties.
As we explain below, it does not suffice to simply take derivatives along the tangent plane, as those miss the higher order data of the variety.
The second is to generalize derivatives from the reals to general fields.
Since we are working with polynomials, differentiation can be viewed as a formal algebraic operation. To handle fields of positive characteristics, we use Hasse derivatives.

Let $V$ be a $k$-dimensional variety in $\FF^n$. Let $I(V)$ be the ideal of polynomials in $\FF[x_1, \dots, x_d]$ that vanish on $V$. Define $R_V = \FF[x_1, \dots, x_d]/I(V)$. The elements of $R_V$ are called \emph{regular functions} on $V$.
Let $p$ be a regular point on $V$, that is, a point where the Zariski tangent space of $V$ at $p$ is also $k$-dimensional.
Given a nonnegative integer $r$, we would like to write down derivative operators $D$ on $\FF[x_1, \dots, x_d]$ so that $Dg(p)$ is well defined not just when $g \in \FF[x_1, \dots, x_d]$, but also when $g$ is a regular function on $V$.
The point here is that regular functions on $V$ may be represented as polynomials in $\FF[x_1, \dots, x_d]$ in non-unique ways (by adding a polynomial that vanishes on $V$), but we should study derivative operators $D$ whose evaluation $Dg(p)$ does not depend on this representation of $g$.

\subsection{An explicit example}

We consider the explicit example of the circle $V$ in $\RR^2$ centered at $(0,1/2)$ of radius $1/2$. In particular, $V$ is defined by the equation $y = x^2 + y^2$. 
Let $p=(0,0)$ be the origin. 
How should we define a second-order derivative at $p$ along $V$? 

Naively one might take $\partial^2 / \partial x^2$ since the tangent at $p$ is the $x$-coordinate direction.
However, consider evaluation of this derivative at $p$ applied to the two sides of $y = x^2 + y^2$ (an identity of regular functions on $V$): the left-hand side gives $0$ while the right-hand side gives 2. So $\partial^2 / \partial x^2$ does not induce a linear functional on the space of regular functions on $V$.

To fix this issue, we can rewrite all regular functions on $V$ as power series centered at $p$ using the \emph{local coordinate} $x$ of $V$. Indeed, by repeated substituting $y \gets x^2 + y^2$, we can write $y$ as a power series in $x$:
\begin{align*}
	y &= x^2 + y^2 \\
	  &= x^2 + (x^2 + y^2)^2 \\
	  &= x^2 + (x^2 + (x^2 + y^2)^2)^2 \\
	  &= x^2 + x^4 + 2x^6 + \cdots 
\end{align*}

We would like a derivative operator $D$ on $\RR[x,y]$ so that $Dg(0,0)$ equals to the coefficient of $x^2$ in $g(x, x^2 + x^4 + 2x^6 + \cdots)$, which in turn equals to the coefficient of $x^2$ plus the coefficient of $y$ in $g(x,y)$.
It is not hard to see that only such choice is $\frac{1}{2}\frac{\partial^2}{\partial x^2} + \frac{\partial }{\partial y}$. Conversely, it is not hard to check that $Dg(0,0) = 0$ for every $g \in \RR[x,y]$ that vanishes identically on $V$.

Elaborating on this example further, 
for each nonnegative integer $r$, we will define $\totalD_{p,V}^r$ to be a one-dimensional space spanned a derivative operator $D$ on $\RR[x,y]$ such that $Dg(0,0)$ equals to the coefficient of $x^r$ in $g(x, x^2 + x^4 + 2x^6 + \cdots)$.
Thus (here $\ang{\cdot}$ denotes the span)
\begin{align*}
	\totalD_{p,V}^0 &= \ang{\operatorname{Id}} \\
	\totalD_{p,V}^1 &= \ang{\frac{\partial}{\partial x}} \\		
	\totalD_{p,V}^2 &= \ang{\frac{1}{2}\frac{\partial^2}{\partial x^2} + \frac{\partial }{\partial y}} \\		
	\totalD_{p,V}^3 &= \ang{\frac{1}{6}\frac{\partial^3}{\partial x^3} + 
							\frac{\partial^2 }{\partial x \partial y}} \\	
	\totalD_{p,V}^4 &= \ang{\frac{1}{24}\frac{\partial^4}{\partial x^4} + 
							\frac{1}{2}\frac{\partial^3 }{\partial x^2 \partial y} + 
							\frac{1}{2}\frac{\partial^2}{\partial y^2}} \\
	&\vdots 
\end{align*}
Then, for each each $D \in \totalD_{p,V}^r$, the map sending $g \in \RR[x,y]$ to $D g(0,0)$ passes to a linear functional on the space $R_V = \RR[x,y]/I(V)$ of regular functions on $V$.

The computation in the above example can be extended to any variety over any field, as we explain below.

\subsection{Local coordinates}

Given a regular point $p$ on a $k$-dimensional variety $V$, after a translation and a linear change of coordinates, suppose that $p$ is at the origin and the first $k$ coordinate vectors are tangent to $V$.
Then by assumption, there are polynomials $f_{k+1},\ldots,f_d$ without any constant or linear terms so that on $V$, we have $x_{k+1} = f_{k+1}(x_1, \dots, x_k)$, $\dots$, $x_d = f_d(x_1, \dots, x_k)$. 
For each $i=k+1, \dots, d$, by repeated substitutions using the defining equations, as functions on $V$, we can write each $x_i$ as a formal power series $h_i(x_1, \dots, x_k)$ in the \emph{local coordinates} $x_1, \dots, x_k$ for $V$ at $p$.

The procedure of taking a power series described earlier can be described in algebraic geometry as a \emph{completion}.
We give a quick summary here and refer the reader to a standard algebraic geometry textbook, e.g., \cite[Chapter 7]{Eisenbud95} \cite[Chapter 29]{Vakil17}.
Let $p$ be a regular point on a $k$-dimensional variety $V$ in $\FF^d$.
Let $\mathfrak{m}_p \subset R_V$ be the maximal ideal of regular functions that vanish at $p$.
Then the \emph{completion} $\widehat{R_{p,V}}$ of $R_V$ at $p$ is the inverse limit $\varprojlim R_{V}/\mathfrak{m}^m_p$. 
The family of projection maps $R_V\to R_{V}/\mathfrak{m}^m_p$ induces a map $\iota_{p,V}:R_V\to \widehat{R_{p,V}}.$

The completion should be thought of as the ring of formal power series around $p$. For example, when $R_V=\FF[x]$ and $\mathfrak{m}_p=(x)$, the completion is the ring of formal power series $\FF\sqbb{x}$. 
More generally, for a regular point $p$ on $V$, assuming that $p$ is the origin and $x_1, \dots, x_k \in \mathfrak m_p$ 
span the Zariski cotangent space $\mathfrak{m}_p/\mathfrak{m}_p^2$, 
the map $\FF\sqbb{x_1,\ldots,x_k}\to \widehat{R_{p,v}}$ sending $x_i$ to $\iota_{p,V}(x_i)$ is an isomorphism (say, by the Cohen structure theorem). 
In other words, there is a \textit{local coordinate system} at $p$ so that every regular function on $V$ can be written as a formal power series around $p$.

It will be useful to know that the formal power series expansion of a regular function is zero if and only if the regular function is zero, i.e., the completion map $R_V \to \wh{R_{p,V}}$ is injective. 
This fact follows from the Krull intersection theorem below (recall that our varieties are always irreducible).

\begin{theorem}[Krull intersection theorem] 
Let $R$ be an integral domain and $I$ be a proper ideal of $R$. Then $\bigcap_{m=0}^{\infty}I^m=\{0\}$.
\end{theorem}

\subsection{Hasse derivatives}

In the explicit example earlier, the main goal of taking derivatives is to extract coefficients.
This is a formal algebraic procedure that does not rely on real analysis.
To allow for arbitrary fields, including those of positive characteristics, we use an algebraic variant known as Hasse derivatives, whose definition and basic properties we summarize below. For proofs of these basic properties of Hasse derivatives, we refer the reader to \cite{DKSS13}, where Hasse derivatives were used to study the finite field Kakeya problem.

\begin{definition}[Hasse derivatives]
For any $d$-tuple $\vec{\omega}=(\omega_1,\ldots,\omega_d)$ of nonnegative integers, 
	define $\Hasse^{\vec{\omega}}$ to be the linear operator on $\FF[x_1,\ldots,x_d]$ given by (writing $x^\delta = x_1^{\delta_1} \cdots x_d^{\delta_d}$ and $\binom{\vec{\gamma}}{\vec{\omega}}:=\binom{\gamma_1}{\omega_1}\cdots \binom{\gamma_d}{\omega_d}$)
\[
	\Hasse^{\vec{\omega}}x^{\vec{\delta}} = \binom{\vec{\delta}}{\vec{\omega}} x^{\vec{\delta}-\vec{\omega}}
\]
for every $d$-tuple $\vec{\delta} = (\delta_1, \dots, \delta_d)$ of nonnegative integers. 
\end{definition}

In particular, $\Hasse^{\vec{\omega}}x^{\vec{\delta}} = 0$ unless $\vec\delta \ge \vec\omega$ coordinatewise.

Over the reals, it is not hard to see that the two notions of derivatives are related by a constant factor
\[
\Hasse^{\vec{\omega}} = \frac{1}{\vec\omega !} \frac{\partial^{\vec\omega}}{\partial x^{\vec \omega}} := \frac{1}{\omega_1!\cdots\omega_d!}\frac{\partial^{\omega_1+\cdots+\omega_d}}{\partial x_1^{\omega_1}\cdots \partial x_d^{\omega_d}}.
\]
Like usual derivatives, Hasse derivatives commute:
\[
\Hasse^{\vec \alpha} \Hasse^{\vec \beta} 
= \binom{\vec\alpha + \vec\beta}{\vec\alpha} \Hasse^{\vec \alpha + \vec\beta} 
= \Hasse^{\vec \beta} \Hasse^{\vec \alpha}.
\]

Hasse derivatives form an algebraic generalization of the usual derivatives when acting on polynomials or formal power series.
The evaluation of a Hasse derivative corresponds to coefficient extraction (without the factorial factors that might be troublesome in fields of positive characteristics). 
Indeed, we have the following ``Taylor's theorem'': given formal variables $x_1, \dots, x_d, y_1, \dots, y_d$ and a polynomial $g \in \FF[x_1, \dots, x_d]$, we have
\begin{equation}\label{eq:taylor}
    g(x+y) = \sum_{\vec{\omega}\in\ZZ_{\geq 0}^d}(\Hasse^{\vec{\omega}}g)(x)y^{\vec{\omega}}
\end{equation}
for any $g\in\FF[x_1,\ldots,x_d]$. 
This identity can be easily checked for each monomial $g(x) = x^{\vec \delta}$.
From this characterization, we see that Hasse derivatives behave well under affine coordinate transforms (as we would expect for derivatives). 
For example, it makes sense to talk about directional Hasse derivatives without specifying a choice of a coordinate system.

\subsection{Higher order directional derivatives}

Now that we have the tools of completion and Hasse derivatives, we are ready to define higher order directional derivatives at a regular point $p$ along a $k$-dimensional variety $V$ in $\FF^d$, generalizing the notion for flats from \cref{sec:joints-of-planes}.

By an affine change of coordinates, assume that $p$ is at the origin, and the tangent space of $V$ at $p$ is spanned by the first $k$ coordinate directions. 
For each $i=k+1, \dots, d$, write each $x_i$ as a formal power series $h_i(x_1, \dots, x_k)$ in the ``local coordinates'' $x_1, \dots, x_k$ for $V$ at $p$. Equivalently, $h_i(x_1, \dots, x_k)$ is the image of $x_i$ under the completion map $R_V \to \wh{R_{p,V}} \cong \FF\sqbb{x_1, \dots, x_k}$.

We define $\totalD_{p,V}^r$ to be the space of all linear combinations $D$ of Hasse derivative operators on $\FF[x_1, \dots, x_d]$ such that the map $\FF[x_1, \dots, x_d] \to \FF$ defined by $g \mapsto Dg(p)$ equals a linear form on coefficients of the homogeneous degree $r$ part of 
\[
\wh g(x_1, \dots, x_k) := g(x_1, x_2, \dots, x_k, h_{k+1}(x_1,\dots, x_k), \dots, h_{d}(x_1,\dots, x_k)),
\]
which is the power series representation of $g$ as a regular function on $V$ in local coordinates at $p$.
Let us also write out this definition more explicitly. Given $(\gamma_1, \dots, \gamma_k) \in \ZZ_{\ge 0}^k$, define
\[
D^{\vec{\gamma}}_{p,V}=\sum_{\vec{\omega}\in \ZZ_{\geq 0}^{d}}c^{\vec{\gamma}}_{\vec{\omega}}\Hasse^{\vec{\omega}}
\]
where
\[
c^{\vec{\gamma}}_{\vec{\omega}} = 
\text{the coefficient of }
x_1^{\gamma_1} \cdots x_k^{\gamma_k} 
\text{ in }
x_1^{\omega_1}\cdots x_k^{\omega_k} h_{k+1}(x_1, \dots, x_k)^{\omega_{k+1}} \cdots h_d(x_1, \dots, x_k)^{\omega_d}
\]
Then $D_{p,V}^{\vec \gamma} g(0, \dots, 0)$ equals the coefficient of $x^{\vec\gamma}$ in $\wh g(x_1, x_2, \dots, x_k)$.
We then set 
\[
\totalD_{p,V}^r = \Span \set{ D_{p,V}^{\vec \gamma} : \vec\gamma = (\gamma_1, \dots, \gamma_k)\in\ZZ_{\ge 0}^k, \gamma_1 + \cdots + \gamma_k = r}.
\]
Note that the $D_{p,V}^{\vec \gamma}$ in the above set are linearly independent. To see this, first note that because no $h_i$ has constant or linear terms, one has
\begin{equation} \label{eq:D-top}
	D_{p,V}^{\vec\gamma} \in \Hasse^{(\vec\gamma,0,\ldots,0)} + \Span\set{ \Hasse^{\vec \omega} : \omega_1 + \cdots + \omega_d < \gamma_1 + \cdots + \gamma_k}.
\end{equation}
The Hasse derivative operators $\Hasse^{\vec{\omega}}$ are linearly independent as $\vec\omega$ ranges over $\ZZ_{\ge0}^d$. 
Since the top weight component of $D_{p,V}^{\vec\gamma}$ is $\Hasse^{(\vec\gamma,0,\ldots,0)}$, we see that the $D_{p,V}^{\vec\gamma}$'s are linearly independent as $\vec \gamma$ ranges over $\ZZ_{\ge 0}^k$.

The key property, as well as the motivation for the above definition, is that for every $D \in \totalD_{p,V}^r$, there is a well defined map $R_V \to \FF$ given by $g\mapsto Dg(p)$. To define this derivative evaluation, we can replace $g \in R_V$ by a representative $g \in \FF[x_1, \dots, x_d]$, and we need to check that $Dg(p)$ does not depend on the choice of the representative. 
Indeed, if $g$ is identically zero on $V$, then $\wh g  = 0$, and hence $Dg(p) = 0$.

The above explicit formula defines $\totalD_{p,V}^r$ assuming that $p$ is at the origin and the tangent space of $V$ at $p$ is spanned by the first $k$ coordinate directions.
By an affine transformation (using \cref{eq:taylor} to determine the behavior of Hasse derivatives under affine transformations), we can define the space $\totalD_{p,V}^r$ of $r$-th order directional derivatives at any regular point $p$ on a variety $V$. 

Having defined $\totalD_{p,V}^r$, we now can proceed nearly identically as in \cref{sec:joints-of-planes} to prove the joints theorem for varieties.
Details are given in the next section.

\section{Proof of the main theorem}\label{sec:carbery}

\subsection{Priority order, handicaps, and a choice of basis}

Given a set of joints $\cJ$ with a fixed preassigned order, and a handicap $\vec\alpha \in \ZZ^\cJ$, we define the priority order $\prec$ on $\cJ \times \ZZ_{\ge 0}$ as before.

Let $n$ be a positive integer.
Let $R_{V,\leq n}$ denote the space of regular functions on $V$ that can be represented as a polynomial of degree at most $n$ in $x_1, \dots, x_n$.
In other words, $R_{V,\leq n}$ is the image of $\FF[x_1,\ldots,x_d]_{\leq n}$ under 
the projection $\FF[x_1,\ldots,x_d] \to R_V$. 

Define $\totalB^r_{p,V}(n)$ to be the set of linear functionals on $R_{V,\le n}$ of the form $g \mapsto Dg(p)$ for some $D \in \totalD_{p,V}^r$ (this is a well defined linear functional as explained earlier). Note that $g \in R_{V, \le n}$ vanishes under $\totalB_{p,V}^0(n) + \cdots + \totalB_{p,V}^{r-1}(n)$ if and only if $g$ vanishes to order at least $r$ at $p$. Here a regular function $g$ on $V$ \emph{vanishes at $p$ to order at least $r$} if $g\in \mathfrak{m}_{p,V}^r$ where $\mathfrak{m}_{p,V}$ is the maximal ideal of $R_V$ corresponding to $p$. Equivalently, power series representation of $g$ using local coordinates at $p$ has no terms with degree lower than $r$.

Now, exactly as in \cref{sec:deriv-eval}, we go through all pairs $(p,r) \in (\cJ \cap V)\times \ZZ_{\ge 0}$ according to the priority order and choose sets $\prioB_{p,V}^r(\vec \alpha, n) \subset \totalB_{p,V}^r(n)$ as earlier so that the disjoint union 
$\bigcup_{(p',r') \preceq (p,r)} \prioB_{p',V}^{r'}(\vec \alpha, n)$
is a basis of 
$\sum_{(p',r') \preceq (p,r)} \totalB_{p',V}^{r'}(\vec \alpha, n)$.
Choose
$\prioD_{p,V}^r (\vec\alpha, n) \subset \totalD_{p,V}^r$ 
with the same size as
$\prioB_{p,V}^r(\vec \alpha, n)$ so that 
$
\prioB_{p, V}^r(\vec \alpha, n) = \{g \mapsto Dg(p) : D \in \prioD_{p, V}^r(\vec \alpha, n)\}
$.
Finally, write 
$\prioB_{p, V}(\vec \alpha, n) := 
\bigcup_{r \ge 0} \prioB_{p, V}^r(\vec \alpha, n)$
and
$\prioD_{p, V}(\vec \alpha, n) := \bigcup_{r \ge 0} \prioD_{p, V}^r(\vec \alpha, n)$.

From the Krull intersection theorem, it follows that for every $p \in V$, $\sum_{r \ge 0} \totalB_{p,V}^r(\vec \alpha, n)$ spans the dual space of $R_{V, \le n}$. 
Hence the disjoint union $\bigcup_{p \in V} \prioB_{p, V}(\vec \alpha, n)$ is a basis of the space of linear forms on $R_{V, \le n}$.
Thus
\begin{equation}\label{eq:sum-var}
\sum_{p\in \cJ\cap V}\abs{\prioB_{p, V}(\vec{\alpha},n)}
=
\dim R_{V, \le n}
=
\deg V \binom{n}{\dim V} + O_V(n^{\dim V - 1}).
\end{equation}
Furthermore there is some $n_0(V)$ so that $\dim R_{V, \le n}$ is a polynomial in $n$ for all $n\ge n_0(V)$. This is a standard fact about the Hilbert series for a variety (see, e.g., \cite[Chapter 18.6]{Vakil17}).

\subsection{Regular functions with given vanishing orders}

This subsection parallels \cref{sec:poly-vanish}.
Here we fix a $k$-dimensional variety $V$ and a finite set of points $\cP \subset V$.
Given a vector $\vec v \in \ZZ_{\ge 0}^\cP$, 
define 
\[
\totalT(\vec v, n)  = \{ g \in R_{V, \le n} : g \text{ vanishes to order} \ge v_p \text{ at each } p\in\cP\}.
\]
Set $b_p(\vec v, n) := \codim_{\totalT(\vec v, n)} \totalT(\vec v + \vec e_p, n)$.

\begin{lemma}[Bounded domain] \label{lem:var-Tzero}
For every $n$ there is some $C_V(n)$ 
so that if $\vec v \in \ZZ_{\ge 0}^\cP$ has $\max_{p \in \cP} v_p > C_V(n)$ then $\dim \totalT(\vec v, n) = 0$. 
\end{lemma}

\begin{proof}
By the Krull intersection theorem, 
$
	\bigcap_{m \ge 0} \mathfrak{m}_{p,V}^m=\{0\}
$.
Since $R_{V,\leq n}$ is finite dimensional, there exists $C = C_V(n)$ such that $\mathfrak{m}_{p,V}^C\cap R_{V,\leq n}=\{0\}$. Hence, if $\vec{v}\in\ZZ_{\geq 0}^{\cJ\cap V}$ satisfies $v_p\geq C$, then $\totalT_V(\vec{v},n)=\{0\}$.
\end{proof}

We omit the proofs of the next two lemmas, which mirror those of \cref{sec:poly-vanish}, except to note that the last line of the proof of \cref{lem:Tlip2} should be adapted as
\[
\codim_{\totalT(\vec 0, n)} \totalT(\vec 0, n-1)
= \dim R_{V, \le n} - \dim R_{V, \le n-1} 
= \deg V \binom{n}{\dim V -1} + O_V(n^{\dim V - 2}).
\]
To see this we use the fact that $\dim R_{V, \le n}$, for sufficiently large $n$, equals to a polynomial (the Hilbert polynomial) whose leading term given in \cref{eq:sum-var}. The right-hand side is the finite difference of this polynomial which can readily be seen to have the above form.

\begin{lemma}[Monotonicity] \label{lem:var-Tmono}
Let $p \in \cP$. Suppose $\vec v^{(1)}, \vec v^{(2)} \in \ZZ_{\ge 0}^\cP$ satisfy $\vec v^{(1)} \ge \vec v^{(2)}$ coordinatewise and with equality at $p$.
Then $b_p(\vec v^{(1)}, n) \le b_p(\vec v^{(2)},n)$ for all $n$.
\end{lemma}

\begin{lemma}[Lipschitz continuity] \label{lem:var-Tlip}
Let $p,q\in\cP$ be distinct points. Suppose $\vec v^{(0)}, \vec v^{(1)}, \dots \in \ZZ^\cP$ are such that $\vec v^{(0)} \le  \vec v^{(1)} \le \cdots$ coordinate-wise and strictly increasing at the coordinate indexed by $p$. Then
\[
0\le \sum_{r \ge 0} b_p(\vec v_p^{(r)}, n)
- \sum_{r \ge 0} b_p(\vec v_p^{(r)} + \vec e_q, n) 
\le \deg V \binom{n}{\dim V -1} + O_V(n^{\dim V - 2}).
\]
\end{lemma}

\subsection{How the number of vanishing conditions varies with the handicap} \label{sec:vanish-handicap-var}

The lemmas in \cref{sec:vanish-handicap} can now be easily adapted to varieties. As in the previous subsection, we continue to focus our attention on a set of points $\cP$ on a variety $V$.

Given a handicap $\vec\alpha \in \ZZ^\cP$ (restricted to $V$), we define the vector $\vec{v}^{p,r}(\vec{\alpha})$ identically to \cref{sec:vanish-handicap-var}.
We have 
\[
\abs{\prioB_{p,V}^r(\vec \alpha, n)}
= b_p(\vec v, n)
= \codim_{\totalT(\vec v, n)} \totalT(\vec v + \vec e_p, n)
\quad \text{ with } 
\vec v = \vec v^{p,r}(\vec \alpha).
\]

We omit the proofs of the following lemmas, which mirror those of \cref{sec:vanish-handicap} but now using the lemmas from the previous subsection.

\begin{lemma}[Bounded domain] \label{lem:var-zero}
For each $n$ there is some $C_V(n)$ so that if $p\in \cP$ and $\vec\alpha \in \ZZ^\cP$ satisfy $ \alpha_p < \max_{q \in \cP} \alpha_q - C_V(n)$,
then $\abs{\prioB_{p,V}(\vec\alpha, n)}=0$. 
\end{lemma}

\begin{lemma}[Monotonicity] \label{lem:var-mono}
Let $n$ be a positive integer and $\vec\alpha^{(1)},\vec\alpha^{(2)} \in \ZZ^\cP$ be two handicaps. 
Suppose $p \in \cP$ satisfies 
$\alpha^{(1)}_p-\alpha^{(1)}_{p'}\le \alpha^{(2)}_p-\alpha^{(2)}_{p'}$ for all $p' \in \cP$.
Then $\abs{\prioB_{p,V}(\vec\alpha^{(1)},n)} \le \abs{\prioB_{p,V}(\vec\alpha^{(2)},n)}$.
\end{lemma}

\begin{lemma}[Lipschitz continuity]\label{lem:var-lip}
Let $p \in \cP$. Let $\vec\alpha^{(1)},\vec\alpha^{(2)} \in \ZZ^\cP$.
Then
\begin{multline*}
\abs{
\sabs{\prioB_{p, V}(\vec\alpha^{(1)},n)}
- \sabs{\prioB_{p, V}(\vec\alpha^{(2)},n)}
}
\\
\le 
\paren{\deg V \binom{n}{\dim V -1} + O_V(n^{\dim V - 2})}\sum_{p'\in \cP }\abs{(\alpha^{(1)}_{p'}-\alpha^{(1)}_{p})-(\alpha^{(2)}_{p'}-\alpha^{(2)}_{p})}.	
\end{multline*}
\end{lemma}

\subsection{Joints configuration} \label{sec:config}

We are ready to discuss joints of varieties. 
Here we set some notation and definitions.

By a \emph{$(k_1, \dots, k_r; m_1, \dots, m_r)$-joints configuration} (or just a \emph{joints configuration} for short) we mean a tuple $(\cJ, \cV_1, \dots, \cV_d)$ as in \cref{thm:main}, namely that each $\cV_i$ is a finite multiset of $k_i$-dimensional varieties in $\FF^d$, where $d = m_1k_1 + \cdots + m_rk_r$,
and $\cJ$ is the set of joints formed by choosing $m_i$ elements from $\cV_i$ for each $i = 1, \dots, r$.
We write $\cM(p)$ for the multiset of $r$-tuples $(\cS_1, \dots, \cS_r)$, where each $\cS_i$ is an unordered $m_i$-tuple of elements of $\cS_i$ and such that together these $s = m_1 + \cdots + m_r$ varieties form a joint at $p$. The quantity $M(p)$ from \cref{thm:main} is then the cardinality of $\cM(p)$. 
We have $M(p) > 0$ at each $p \in \cJ$.

\subsection{Vanishing lemma}

Before stating the analog to \cref{lem:vanishing}, let us first note the following observation about how high order directional derivatives of several varieties interact at a joint.

\begin{lemma}\label{lem:hasse-vanishing}
Let $p$ be a joint formed by varieties $V_1,\ldots,V_s$.
Suppose $g\in\FF[x_1,\ldots,x_d]$ vanishes to order exactly $r$ at $p$ (as a polynomial function on $\FF^d$).
Then there exist $r_1, \dots r_s \in \ZZ_{\ge 0}$ with $r_1+\cdots+r_s=r$ and $D_1\in \totalD^{r_1}_{p,V_1},\ldots,D_s\in \totalD^{r_s}_{p,V_s}$ such that
\[\left(D_1D_2\cdots D_s g\right)(p) \neq 0.\]
\end{lemma}
\begin{proof}
Let $k_i = \dim V_i$ for each $i$.
By an affine change of coordinates, suppose that $p$ is at the origin, $V_1$ is tangent to the first $k_1$ coordinate vectors, $V_2$ tangent to the next $k_2$ coordinate vectors, and so on.
Let $c x_1^{\gamma_1} \cdots x_d^{\gamma_d}$, $c \in \FF\setminus \{0\}$, be a monomial of lowest degree in $g$. 
Since $g$ vanishes to order $r$ at $p$, we have $\gamma_1 + \cdots + \gamma_d = r$.
Let $r_1$ be the sum of the first $k_1$ $\gamma_i$'s, 
$r_2$ the sum of the next $k_2$ $\gamma_i$'s, and so on.
By \cref{eq:D-top}, there exist $D_i\in \totalD^{r_i}_{p,V_i}$ of the form
\begin{align*}
	D_1 &= \Hasse^{(\gamma_1, \dots, \gamma_{k_1}, 0, \dots, 0)} + \text{lower order derivatives},\\
	D_2 &= \Hasse^{(0, \dots, 0, \gamma_{k_1+1}, \dots, \gamma_{k_1+k_2}, 0, \dots, 0)} + \text{lower order derivatives},\\
	&\dots.
\end{align*}
Then $D_1D_2\cdots D_s g = c + \text{higher order terms}$, which evaluates to $c\ne 0$ at $p = 0$.
\end{proof}

The next statement is analogous to the vanishing lemma for planes in \cref{lem:vanishing}.
The proof is analogous, but we write it out explicitly here since it is a critical step of the argument.

\begin{lemma}\label{lem:carbery-vanishing}
	Let $(\cJ, \cV_1,\ldots, \cV_k)$ be a $(k_1, \dots, k_r; m_1, \dots, m_r)$-joints configuration. 
	Let $s = m_1 + \cdots + m_r$ and $d = m_1k_1 + \cdots + m_rk_r$.
	Fix a handicap $\vec \alpha$ and its associated priority order. 
	Fix a positive integer $n$. 
	Choose $\prioD_{p,V}$ as earlier.
	For each $p \in \cJ$, fix a choice $V_1(p), V_2(p), \dots, V_s(p)$ of varieties that form a joint at $p$, and of which exactly $m_i$ of them come from $\cV_i$ for each $i = 1, \dots, r$.

	Then for every nonzero polynomial $g \in \FF[x_1, \dots, x_d]$ of degree at most $n$, one has 	
	$D_1\cdots D_s g(p) \ne 0$ 
	for some joint $p \in \cJ$  
	and some $D_1 \in \prioD_{p, V_1(p)}$, \dots, $D_s \in \prioD_{p, V_s(p)}$.
\end{lemma}

\begin{proof}
	Suppose, on the contrary, 
	that there were some nonzero polynomial $g \in \FF[x_1, \dots, x_d]$ of degree at most $n$
	such that
	$D_1\cdots D_s g(p) = 0$ 
	for every joint $p \in \cJ$ and $D_1 \in \prioD_{p, V_1}$, \dots, $D_s \in \prioD_{p, V_s}$,
	where $V_1, V_2, \dots, V_s$ are any varieties that form a joint at $p$ and exactly $m_i$ of them come from $\cV_i$ for each $i = 1, \dots, r$, 
	
	Choose $p \in \cJ$ to minimize $(p, v_p(g))$ under $\prec$, where $v_p(g)$ is the order vanishing of $g$ at $p$. 
	
	Since $g$ vanishes to order exactly $v_p(g)$ at $p$, by \cref{lem:hasse-vanishing},  there exist 
	$D_1 \in \totalD_{p, V_1(p)}^{r_1}$, \dots, $D_s \in \totalD_{p, V_s(p)}^{r_s}$
	with $D_1D_2\cdots D_s g(p) \ne 0$
	and $r_1 + \cdots + r_s = v_p(g)$. 
	Among all choices of $D_1, \dots, D_s$ (including choices of $r_1, \dots, r_s$), choose ones so that $|\{ i \in [s] : D_i \in \prioD_{p, V_i(p)}\}|$ is maximized.	
	By the assumption at the beginning of the proof, one must have $D_i \notin \prioD_{p, V_i}$ for some $i \in [s]$. 
	Relabeling if necessary, assume that $D_1 \notin \prioD_{p, V_1(p)}$. (Here we are using that derivatives commute.)
	
	Suppose $p' \in V_1(p) \cap \cJ$ and $r' \in \ZZ_{\ge 0}$ satisfy $(p', r') \prec (p, r_1)$. 
	We get $(p', r' + r_2 +  \cdots + r_s) \prec (p, r_1 + r_2 + \cdots + r_s) = (p, v_p(g))$. 
	By the choice of $p$, we have $(p, v_p(g)) \preceq (p', v_{p'}(g))$.
	Thus $(p', r' + r_2 + \cdots + r_s) \prec (p', v_{p'}(g))$, and hence $r' + r_2 + \cdots + r_s < v_{p'}(g)$.
	If follows that $D D_2 \cdots D_s g (p') = 0$ for all $D \in \totalD_{p', V_1(p)}^{r'}$ by the definition of vanishing order.
	
	From the above paragraph we deduce that $D_2\cdots D_s g(p')$ lies in the common kernel of $\totalB_{p', V_1(p)}^{r'}$ ranging over all $(p',r') \in (V_1(p) \cap \cJ) \times \ZZ_{\ge 0}$ with $(p',r') \prec (p,r_1)$.
	Since $D_1D_2\cdots D_s g(p) \ne 0$, we deduce that $D_2\cdots D_s g$ does not lie in the common kernel of $\prioB_{p,V_1(p)}^{r_1}$, i.e., there is some $D \in \prioD_{p, V_1(p)}^{r_1}$ with $DD_2\cdots D_s g(p) \ne 0$. But this $D$ contradicts the earlier assumption that the choice of $D_1, \dots, D_s$ maximizes $|\{ i \in [s]: D_i \in \prioD_{p, V_i(p)}\}|$.
\end{proof}

The next lemma is a consequence of parameter counting. Its proof is identical to that of \cref{lem:param-ineq} except that we now apply \cref{lem:carbery-vanishing}.

\begin{lemma}\label{cor:car-param-ineq}
Assume the same setup as \cref{lem:carbery-vanishing}. We have
\[
\sum_{p\in\cJ} \prod_{i=1}^s \abs{\prioD_{p,V_i(p)}(\vec{\alpha},n)}\geq \binom{n+d}{d}.
\]
\end{lemma}

\subsection{Choosing the handicaps} We say that a joints configuration $(\cJ,\cV_1,\ldots,\cV_r)$ is \emph{connected} if the following graph is connected: the vertex set is $\cJ$, with $p,p'\in \cJ$ adjacent if there is some $V\in \cV_1\cup\cdots\cup \cV_r$ containing both $p$ and $p'$.

\begin{lemma}\label{lem:carbery-handicaps} 
Let $n$ be a positive integer and $(\cJ, \cV_1, \dots, \cV_k)$ be a connected $(k_1, \dots, k_r; m_1, \dots, m_r)$-joints configuration.
Let $\omega(p)$ be a positive real for each $p\in \cJ$. 
Then there exists a choice of handicap $\vec \alpha \in \ZZ^\cJ$ such that
\[
	\frac{1}{\omega(p)}\left[\prod_{(\cS_1, \dots, \cS_r) \in \cM(p)}\prod_{V \in \cS_1 \cup \cdots \cup \cS_r}\frac{\abs{\prioD_{p,V}(\vec{\alpha},n)}}{\binom{n}{\dim V}}\right]^{1/M(p)}
\]
lies in some common interval of length $o_{\cJ, \cV_1, \dots, \cV_k,\omega;n\to\infty}(1)$ as we range over $p\in \cJ$. Here the notation means that the length of the interval tends to zero as $n$ goes to infinity but the rate may depend on the joints configuration and $\omega$.
\end{lemma}

\begin{proof}
The proof is analogous to \cref{lem:handicaps} with appropriate modification. In this proof, we use $o(1)$ to denote $o_{\cJ, \cV_1, \dots, \cV_k,\omega;n\to\infty}(1)$. Let $(\delta_n)_{n\in\NN}$ be a sequence tending to $0$ sufficiently slowly as $n$ tends to infinity. Denote by $W_p(\vec{\alpha})$ the quantity (the dependence of $W_p(\vec{\alpha})$ on $n$ is suppressed in the notation).
\[
\frac{1}{\omega(p)}\left[\prod_{(\cS_1, \dots, \cS_r) \in \cM(p)}\prod_{V \in \cS_1 \cup \cdots \cup \cS_r}\frac{\abs{\prioD_{p,V}(\vec{\alpha},n)}}{\binom{n}{\dim V}}\right]^{1/M(p)}.
\]
We begin by noticing that, by \cref{lem:var-zero}, there exists some $c$ depending on $n$ and the joints configuration such that if $\alpha_p<\alpha_{p'}-c$ for two joints $p,p'$ on the same flat $V$, then $|\prioD_{p,V}(\vec{\alpha},n)|=0$, which shows that $W_p(\vec{\alpha})=0$. Therefore, although there are infinitely many choices for $\vec \alpha \in \ZZ^\cJ$,
there are only finitely many possible values of $(W_p)_{p\in\cJ}$ they can produce for a given $n$. Choose the $W_p(\vec{\alpha})$ so that after sorting $W_p$ in the descending order, it has the least lexicographical order. Suppose that the sorted result is
\[W_{p_1}(\vec{\alpha})\geq \cdots\geq W_{p_{|\cJ|}}(\vec{\alpha}).\]
It suffices to show that $W_{p_i}(\vec{\alpha})-W_{p_{i+1}}(\vec{\alpha})\leq \delta_n$ for all $i=1,\ldots,|\cJ|-1$.

Suppose for the sake of contradiction that the claim fails for some $i$. Let $t$ be the least positive integer such that $W_{p_t}(\vec{\alpha}) - W_{p_{t+1}}(\vec{\alpha}) > \delta_n$. Then let $\vec{v}=\vec e_{p_1}+\cdots+\vec e_{p_t}$ and $\vec{\alpha}'=\vec{\alpha}-\vec{v}$. Take a constant $C'$ larger than all the degrees of the varieties in the joints configuration. Similar to the proof of \cref{lem:handicaps}, we can apply \cref{lem:var-lip} to show that $\abs{\sabs{\prioD_{p,V}(\vec{\alpha},n)}-\sabs{\prioD_{p,V}(\vec{\alpha}',n)}}/\binom{n}{\dim V}=o(1)$. Together with the fact that $\abs{\prioD_{p,V}(\vec{\alpha},n)}/\binom{n}{\dim V}\leq C'+o(1)$ (guaranteed by \cref{eq:sum-var}) we can use a similar telescoping inequality to show that the difference between $W_p(\vec{\alpha},n)$ and $W_p(\vec{\alpha}',n)$ is at most $o(1)$. Therefore the difference is bounded by $\delta_n/2$ as long as $\delta_n$ tends to $0$ slowly enough.

Now, by the new monotonicity established in \cref{lem:var-mono}, 
we know that $W_{p_i}(\vec{\alpha}')\leq W_{p_i}(\vec{\alpha})$ if $i\leq t$, 
and $W_{p_i}(\vec{\alpha}')\geq W_{p_i}(\vec{\alpha})$ if $i>t$. 
If $W_p(\vec{\alpha})\neq W_p(\vec{\alpha}')$ for some $p\in\cJ$, 
then by \cref{eq:sum-var},
we know that there exist $i\leq t$ and $p_i\in V$ such that $\abs{\prioD_{p_i,V}(\vec{\alpha}',n)}<\abs{\prioD_{p_i,V}(\vec{\alpha},n)}$, 
resulting in $W_{p_i}(\vec{\alpha}')<W_{p_i}(\vec{\alpha})$. 
By the fact that $\abs{W_p(\vec{\alpha}')-W_p(\vec{\alpha})}\leq \delta_n/2$ for all $p\in\cJ$ and the assumption that $W_{p_t}-W_{p_{t+1}}>\delta_n$, 
we know that $W_{p_1}(\vec{\alpha}'),\ldots,W_{p_t}(\vec{\alpha}')$ are still the $t$ largest ones among $(W_p(\vec{\alpha}'))_{p\in\cJ}$. 
Hence, that $W_{p_i}(\vec{\alpha}')<W_{p_i}(\vec{\alpha})$ is a contradiction with the assumption of the minimality under the lexicographical order.

The previous paragraph shows that $W_p(\vec{\alpha})=W_p(\vec{\alpha}')$ for every $p\in\cJ$. As a consequence, $W_p(\vec{\alpha})=W_p(\vec{\alpha}-m\vec{v})$ for all positive integers $m$ and $p\in\cJ$. Since the joints configuration is connected, we can find $i\leq t<j$ such that $p_i$ and $p_j$ lie on the same variety. When $m$ is sufficiently large, we have $\alpha_{p_i}-m<\alpha_{p_j}-c$, which forces $W_{p_i}(\vec{\alpha}-m\vec{v})$ to be $0$. By the ordering, this shows that $W_{p_i'}(\vec{\alpha}-m\vec{v}) = 0$ for all $i'\leq i$. In particular, $W_{p_t}(\vec{\alpha})=W_{p_{t+1}}(\vec{\alpha})=0$, contradicting $W_{p_t}(\vec{\alpha}) - W_{p_{t+1}}(\vec{\alpha}) > \delta_n$.
\end{proof}

\begin{proof}[Proof of \cref{thm:main}\ref{part:const-mult}]
In this proof $o(1)$ denotes a quantity which goes zero as $n$ goes to infinity but can dependent arbitrarily on the joints configuration. Similar to the proof of \cref{thm:2-flats}, it suffices to consider the case where the joints configuration is connected. Set 
\[
s=m_1+\cdots+m_r,
\]
and
\[
J_\omega = \sum_{p\in\cJ}\omega(p) \quad \text{where} \quad
\omega(p)=M(p)^{1/(s-1)}.
\]
Choose $\vec{\alpha}$ according to \cref{lem:carbery-handicaps}. Then we can choose $W$ so that
\[\left|\frac{1}{\omega(p)}\left[\prod_{(\cS_1, \dots, \cS_r) \in \cM(p)}\prod_{V \in \cS_1 \cup \cdots \cup \cS_r}\frac{|\prioD_{p,V}|}{\binom{n}{\dim V}}\right]^{1/M(p)}-W\right|=o(1)\]
for all $p\in\cJ$.
Hence, by \cref{cor:car-param-ineq} (choosing $V_i(p)$ of \cref{cor:car-param-ineq} to give the minimum product below), we have that
\[\sum_{p\in\cJ}\omega(p)W\geq\sum_{p\in\cJ}\min_{(\cS_1, \dots, \cS_r) \in \cM(p)}\prod_{V \in \cS_1 \cup \cdots \cup \cS_r}\frac{|\prioD_{p,V}|}{\binom{n}{\dim V}}-o(1)\geq \frac{\binom{n+d}{d}}{\prod_{i=1}^{r}\binom{n}{k_i}^{m_i}}-o(1),\]
which, after rearrangement, shows that 
\[W\geq \frac{\prod_{i=1}^{r}(k_i!)^{m_i}}{J_{\omega}\cdot d!}-o(1).\]

Let $\cV_{p,i}$ be the set of varieties in $\cV_i$ that contain $p$. Then we have that for any joint $p\in \cJ$, 
\begin{align*}
M(p)\omega(p)W\leq& \sum_{(\cS_1, \dots, \cS_r) \in \cM(p)}\prod_{V \in \cS_1 \cup \cdots \cup \cS_r}\frac{\abs{\prioD_{p,V}}}{\binom{n}{\dim V}}+o(1)&& \text{\footnotesize [by AM-GM]}\\
\leq& \sum_{\cS_1\in \binom{\cV_{p,1}}{m_1}, \ldots, \cS_r\in \binom{\cV_{p,r}}{m_r}}\prod_{i=1}^{r}\prod_{V \in \cS_i}\frac{\abs{\prioD_{p,V}}}{\binom{n}{k_i}}+o(1)\\
=&\prod_{i=1}^{r}\sum_{\cS_i\in \binom{\cV_{p,i}}{m_i}}\prod_{V\in \cS_i}\frac{\abs{\prioD_{p,V}}}{\binom{n}{k_i}}+o(1)\\
\leq& \prod_{i=1}^{r}\frac{(\deg \cV_i )^{m_i}}{m_i!}\left(\sum_{V\in \cV_{p,i}}\frac{\abs{\prioD_{p,V}}}{\deg \cV_i \binom{n}{k_i}}\right)^{m_i}+o(1)\\
\leq& \frac{(\deg \cV_1)^{m_1}\cdots(\deg \cV_r)^{m_r}}{m_1!\cdots m_r!s^{s}}\left(\sum_{i=1}^{r}\sum_{V\in \cV_{p,i}}\frac{m_i\abs{\prioD_{p,V}}}{\deg \cV_i\binom{n}{k_i}}\right)^{s}+o(1).&& \text{\footnotesize [by AM-GM]}
\end{align*}
By taking the $s$-th root on both sides, summing over all $p$ using \cref{eq:sum-var} and noticing that $M(p)\omega(p)=\omega(p)^s$, we conclude that
\begin{align*}
\sum_{p\in \cJ}\omega(p)W^{1/s}
&\le \frac{1}{s} \left(\frac{\paren{\deg \cV_1}^{m_1}\cdots \paren{\deg \cV_r}^{m_r}}{m_1!\cdots m_r!}\right)^{1/s}
\sum_{i=1}^{r}\sum_{V\in \cV_i} \sum_{p\in \cJ \cap \cV_i} \frac{m_i\abs{\prioD_{p,V}}}{\deg \cV_i\binom{n}{k_i}}+o(1)
\\
&= \frac{1}{s} \left(\frac{\paren{\deg \cV_1}^{m_1}\cdots \paren{\deg \cV_r}^{m_r}}{m_1!\cdots m_r!}\right)^{1/s}
\sum_{i=1}^{r}\sum_{V\in \cV_i} \frac{m_i \deg V}{\deg \cV_i} +o(1)
\\
&= \frac{1}{s} \left(\frac{\paren{\deg \cV_1}^{m_1}\cdots \paren{\deg \cV_r}^{m_r}}{m_1!\cdots m_r!}\right)^{1/s}
\sum_{i=1}^{r} m_i +o(1)
\\
&=  \left(\frac{\paren{\deg \cV_1}^{m_1}\cdots \paren{\deg \cV_r}^{m_r}}{m_1!\cdots m_r!}\right)^{1/s}+o(1).
\end{align*}
Rearranging, we find
\[W\leq \frac{\paren{\deg \cV_1}^{m_1}\cdots \paren{\deg \cV_r}^{m_r}}{m_1!\cdots m_r!J_{\omega}^s}+o(1).\]
By comparing the lower and upper bounds on $W$, and letting $n \to \infty$ so that the $o(1)$ term vanishes, we have
\[
\frac{\prod_{i=1}^{r}(k_i!)^{m_i}}{J_{\omega}\cdot d!} \le \frac{\paren{\deg \cV_1}^{m_1}\cdots \paren{\deg \cV_r}^{m_r}}{m_1!\cdots m_r!J_{\omega}^s}.
\]
Rearranging gives the desired conclusion
\[
J_\omega \le \paren{ d! \prod_{i=1}^r \frac{(\deg \cV_i)^{m_i}}{k_i!^{m_i} m_i! }}^{1/(s-1)}. \qedhere 
\]
\end{proof}

\begin{proof}[Proof of \cref{thm:main}\ref{part:const-no-mult}]
As earlier, we may assume that the joints configuration is connected. Set $s=m_1+\cdots+m_r$ throughout the proof. Choose $\vec{\alpha}$ according to \cref{lem:carbery-handicaps} with $\omega(p)=1$ for all $p\in\cJ$. Then we can choose $W$ so that
\[\left|\left[\prod_{(\cS_1, \dots, \cS_r) \in \cM(p)}\prod_{V \in \cS_1 \cup \cdots \cup \cS_r}\frac{|\prioD_{p,V}|}{\binom{n}{\dim V}}\right]^{1/M(p)}-W\right|=o(1)\]
for all $p\in\cJ$. Hence, by \cref{cor:car-param-ineq}, we have that
\[\sum_{p\in\cJ}W\geq\sum_{p\in\cJ}\min_{(\cS_1, \dots, \cS_r) \in \cM(p)}\prod_{V \in \cS_1 \cup \cdots \cup \cS_r}\frac{|\prioD_{p,V}|}{\binom{n}{\dim V}}-o(1)\geq \frac{\binom{n+d}{d}}{\prod_{i=1}^{r}\binom{n}{k_i}^{m_i}}-o(1),\]
which, after rearrangement, shows that 
\[W\geq \frac{\prod_{i=1}^{r}(k_i!)^{m_i}}{\abs{\cJ}\cdot d!}-o(1).\]

For each $p\in\cJ$, let $(\cS_1(p), \dots ,\cS_r(p))\in \cM(p)$ be the element $S$ of $\cM(p)$ such that 
\[\prod_{V \in \cS_1 \cup \cdots \cup \cS_r}\frac{\abs{\prioD_{p,V}}}{\binom{n}{\dim V}}\]
is maximized. Then $W\leq \prod_{V \in \cS_1(p) \cup \cdots \cup \cS_r(p)}|\prioD_{p,V}|/\binom{n}{\dim V}+o(1)$, which shows that
\begin{align*}
s\abs{\cJ}W^{1/s}\paren{\deg \cV_1}^{-m_1/s}&\cdots\paren{\deg \cV_r}^{-m_r/s}
\\& \leq s\sum_{p\in\cJ}\left(\prod_{i=1}^{r} m_i^{-m_i}\prod_{V\in \cS_i(p)}\frac{m_i\abs{\prioD_{p,V}}}{\deg \cV_i\binom{n}{k_i}}\right)^{1/s}+o(1)\\
& \leq  \frac{1}{m_1^{m_1/s}\cdots m_r^{m_r/s}}\sum_{p\in\cJ}\sum_{i=1}^{r}\sum_{V\in \cS_i(p)}\frac{m_i\abs{\prioD_{p,V}}}{\deg \cV_i\binom{n}{k_i}}+o(1)&& \text{\footnotesize [by AM-GM]}\\
& \leq \frac{1}{m_1^{m_1/s}\cdots m_r^{m_r/s}}\sum_{i=1}^{r}\sum_{V\in\cV_i}\sum_{p\in\cJ\cap V}\frac{m_i\abs{\prioD_{p,V}}}{\deg \cV_i\binom{n}{k_i}}+o(1)\\
& \leq \frac{1}{m_1^{m_1/s}\cdots m_r^{m_r/s}}\sum_{i=1}^{r}\sum_{V\in\cV_i} \frac{m_i \deg V}{\deg \cV_i}+o(1) &&\text{\footnotesize [by \cref{eq:sum-var}]}\\
& \leq \frac{1}{m_1^{m_1/s}\cdots m_r^{m_r/s}}\sum_{i=1}^{r}m_i+o(1)\\
& = \frac{s}{m_1^{m_1/s}\cdots m_r^{m_r/s}}+o(1).
\end{align*}
By rearranging, we get that
\[W\leq \frac{\paren{\deg \cV_1}^{m_1}\cdots \paren{\deg \cV_r}^{m_r}}{m_1^{m_1}\cdots m_r^{m_r}\abs{\cJ}^s}+o(1).\]
By comparing the lower and upper bounds on $W$, and letting $n \to \infty$ so that the $o(1)$ term vanishes, we have
\[
\frac{\prod_{i=1}^{r}(k_i!)^{m_i}}{\abs{\cJ}\cdot d!}
\leq \frac{\paren{\deg \cV_1}^{m_1}\cdots \paren{\deg \cV_r}^{m_r}}{m_1^{m_1}\cdots m_r^{m_r}\abs{\cJ}^s}.
\]
Rearranging gives the desired conclusion
\[
\abs{\cJ} \le \paren{ d! \prod_{i=1}^r \frac{(\deg V_i)^{m_i}}{{k_i!}^{m_i} m_i^{m_i}}}^{1/(s-1)}. \qedhere 
\]
\end{proof}

\section*{Acknowledgments} This research was conducted while Yu was a participant and Tidor was a mentor in the Summer Program in Undergraduate Research (SPUR) of the MIT Mathematics Department. We thank David Jerison and Ankur Moitra for their role in advising the program. 
We also thank the anonymous referee for a careful reading of the manuscript and helpful comments.

Tidor was supported by the NSF Graduate Research Fellowship Program DGE-1745302.

Zhao was supported by NSF Award DMS-1764176, the MIT Solomon Buchsbaum Fund, and a Sloan Research Fellowship.


\end{document}